\documentclass[graybox,tbtags]{svmult}

\usepackage{mathptmx}       
\DeclareSymbolFont{operators}{OT1}{txr}{m}{n}
\DeclareSymbolFont{italic}{OT1}{txr}{m}{it}
\DeclareSymbolFontAlphabet{\mathrm}{operators}
\DeclareSymbolFont{letters}{OML}{txmi}{m}{it}
\DeclareSymbolFont{symbols}{OMS}{txsy}{m}{n}

\usepackage{helvet}         
\usepackage{courier}        
\usepackage{type1cm}        
\usepackage{makeidx}         
\usepackage{graphicx}        
\usepackage{multicol}        
\usepackage[bottom]{footmisc}
\usepackage{amsmath,amssymb,amsfonts,bm,url,mathrsfs}

\makeindex             

\newbox\shell
\newcommand{\dia}[2]{\setbox\shell=\hbox{\begin{picture}(180,120)(-90,-60)#1
\put(-90,-60){\makebox(180,120)[b]{\large #2}}\end{picture}}\dimen0=\ht
\shell\multiply\dimen0by7\divide\dimen0by16\raise-\dimen0\box\shell}

\newcommand{\vtx}{\circle*{10}}

\DeclareMathOperator{\IKM}{\mathbf{IKM}}

\DeclareMathOperator{\JYM}{\mathbf{JYM}}

\DeclareMathOperator{\tr}{tr}
\DeclareMathOperator{\IM}{Im}
\DeclareMathOperator{\RE}{Re}

\DeclareMathOperator{\Gal}{Gal}
\DeclareMathOperator{\Frob}{Frob}
\DeclareMathOperator{\Tr}{Tr}
\DeclareMathOperator{\Kl}{Kl}
\DeclareMathOperator{\Sym}{Sym}

\usepackage[all]{xy}
\usepackage{enumitem}[2011/09/28]
\setenumerate{align=left, leftmargin=0pt,labelsep=.5em, labelindent=0\parindent,listparindent=\parindent,itemindent=*}

\usepackage{xcolor}
\usepackage[pstarrows]{pict2e}

\usepackage{graphicx}
\usepackage{tikz}

\tikzset{>=stealth}
\usepackage{pgfplots}

\usepackage[english,french,german,russian]{babel}
\begin{document}

\selectlanguage{english}

\title*{Some algebraic and arithmetic properties of Feynman diagrams}
\author{Yajun Zhou}
\institute{Yajun Zhou \at Program in Applied and Computational Mathematics (PACM), Princeton University, Princeton, NJ 08544, USA, \email{yajunz@math.princeton.edu}\at Academy of Advanced Interdisciplinary Studies (AAIS), Peking University, Beijing 100871, People's Republic of China, \email{yajun.zhou.1982@pku.edu.cn}}
%
%
\maketitle

\abstract{This article reports on some recent progresses in Bessel moments, which represent a class of Feynman diagrams in 2-dimensional quantum field theory. Many challenging mathematical problems  on these Bessel moments have been formulated as a vast set of conjectures,  by David Broadhurst and collaborators, who work at the intersection of high energy physics, number theory and algebraic geometry. We present the main ideas behind our verifications of several such conjectures, which revolve around linear and non-linear sum rules of Bessel moments, as well as  relations between individual Feynman diagrams and critical values of modular $L$-functions.  }

\section{Introduction}
\subsection{Bessel moments and Feynman diagrams}In perturbative quantum field theory (pQFT), we use \textit{Feynman diagrams} to quantify the interactions among elementary particles \cite{Groote2007,BBBG2008,Broadhurst2016,Laporta:2017okg}. In this survey, we will focus on 2-dimensional pQFT, where  the  propagator of a free particle with proper mass $m_0$ takes the following form:
\begin{align}
\frac{1}{(2\pi)^{2}}\lim_{\varepsilon\to0^+}\iint_{\mathbb R^2}\frac{e^{i\bm p\cdot\bm x-\varepsilon|\bm p|^{2}}\D^2\bm p}{|\bm p|^2+m^2_{0}}=\frac{K_{0}(m_{0}|\bm x|)}{2\pi}\label{eq:free_propagator}
\end{align}for $ \bm x\in\mathbb R^2\smallsetminus\{\bm 0\}$. Here,  $K_0(t):=\int_0^\infty e^{-t\cosh u}\D u,t>0 $ is the modified Bessel function of the second kind  and zeroth order.

Some results in 2-dimensional pQFT also find their way into the finite part of renormalized perturbative expansions of $ (4-\varepsilon)$-dimensional quantum electrodynamics \cite{Laporta2008}. For example, in Stefano Laporta's recent computation of the 4-loop contribution to  electron's magnetic moment \cite{Laporta:2017okg}, one of the  master integrals is the 4-loop sunrise diagram for 2-dimensional pQFT:\begin{align}\begin{split}{}&
\;\;\;\;\;
\dia{\put(-100,0){\line(1,0){200}}
\put(0,15){\circle{100}}
\put(0,-15){\circle{100}}
\put(50,0){\vtx}
\put(-50,0){\vtx}
}{}\;\;\;:=2^{4}\int_0^\infty I_0(t)[K_0(t)]^5t\D t\\={}&\int_0^\infty\frac{\D x_1}{x_1}\int_0^\infty\frac{\D x_2}{x_2}\int_0^\infty\frac{\D x_3}{x_3}\int_0^\infty\frac{\D x_4}{x_4}\frac{1}{\big(1+\sum^4_{k=1}x_k\big)\big(1+\sum^4_{k=1}\frac{1}{x_{k}} \big)-1}.\end{split}\label{eq:4-loop_sunrise}
\end{align}Here, the single integral over the variable $t$ represents the Feynman diagram in configuration space (see \cite[\S1]{BBBG2008}, \cite[\S9.2]{Broadhurst2013MZV} or \cite[(84)]{Broadhurst2016}), and $I_0(t)=\frac{1}{\pi}\int_0^\pi e^{t\cos\theta}\D\theta$ is the modified Bessel function of the first kind  and zeroth order; alternatively,  a quadruple integral over a rational function in the variables $x_1$, $x_2$, $x_3$ and $x_4$ represents the same Feynman diagram in the Schwinger parameter space (see \cite[\S9.1]{Broadhurst2013MZV} or \cite[\S8]{Vanhove2014Survey}).

On one hand, Feynman diagrams provide us with many physically meaningful multiple integrals over rational functions, which   are special cases of motivic integrals \cite{Vanhove2014Survey,BlochKerrVanhove2015}, playing prominent r\^oles in  the arena for algebraic geometers.
On the other hand, certain Feynman diagrams are (conjecturally or provably) related to arithmetically interesting objects \cite{Samart2016,Broadhurst2016,Zhou2017WEF}, such as special values of modular $L$-functions inside their critical strips, inviting  pilgrims to the pantheon of number theorists.

 After high-precision computations of Feynman diagrams, Bai\-ley--Bor\-wein--Broad\-hurst--Glasser \cite{BBBG2008}, Broadhurst \cite{Broadhurst2013MZV,Broadhurst2016}, Broadhurst--Schnetz \cite{BroadhurstSchnetz2014} and  Broadhurst--Mellit \cite{BroadhurstMellit2016}  had formulated various conjectures on \textit{Bessel moments}\begin{align}
\IKM(a,b;n):=\int_0^\infty [I_0(t)]^a[K_0(t)]^bt^n\D t
\end{align}with $ a,b,n\in\mathbb Z_{\geq0}$. The last few years had witnessed rapid progress in these conjectures proposed by David Broadhurst and coworkers. In \S\S\ref{subsec:alg_BM}--\ref{subsec:arith_BM} below, we give precise statements of some recently proven conjectures about Bessel moments, before presenting in \S\ref{subsec:Leitfaden}  a road map for their mathematical understanding. \subsection{Some algebraic relations involving Bessel moments\label{subsec:alg_BM}}

The following theorem about linear sum rules for Bessel moments grew out of numerical conjectures by Bai\-ley--Bor\-wein--Broad\-hurst--Glasser \cite[(220)]{BBBG2008}, Broadhurst--Mellit \cite[(7.10)]{BroadhurstMellit2016} and Broadhurst--Roberts \cite[Conjecture 2]{Broadhurst2017Paris}. The first proof appeared in \cite{HB1}.\begin{theorem}[Generalized Bailey--Borwein--Broadhurst--Glasser sum rules and generalized  Crandall numbers]\label{thm:B3G_Crandall}\begin{enumerate}[leftmargin=*,  label=\emph{(\alph*)},ref=(\alph*),
widest=a, align=left]\item\label{itm:B3G} We have \begin{align}
\int_0^\infty \frac{[\pi I_{0}(t)+iK_0(t)]^m+[\pi I_{0}(t)-iK_0(t)]^m}{i}[K_0(t)]^mt^n\D t=0\label{eq:HB_sum_rule}
\end{align} for  $ m\in\mathbb Z_{>1},n\in\mathbb Z_{\geq0},\frac{m-n}{2}\in\mathbb Z_{>0}$,   and \begin{align}\label{eq:HB_sum_rule'}
\int_0^\infty \frac{[\pi I_{0}(t)+iK_0(t)]^m-[\pi I_{0}(t)-iK_0(t)]^m}{i}[K_0(t)]^mt^n\D t=0
\end{align}for $ m\in\mathbb Z_{>0},n\in\mathbb Z_{\geq0},\frac{m-n-1}{2}\in\mathbb Z_{>0}$, which generalize the Bailey--Borwein--Broadhurst--Glasser sum rule  \cite[(220)]{BBBG2008}.
\item\label{itm:Crandall} The Crandall numbers (OEIS \texttt{A262961} \cite{OEISA262961})\begin{align}
A(n):={}&\left(\frac{2}{\pi}\right)^4
\int_0^\infty\left\{[\pi I_0(t)]^2 - [K_0(t)]^2\right\} I_0(t)[K_0(t)]^5\,(2t)^{2n-1} \D t
\label{eq:An}
\end{align}are  integers for all $ n\in\mathbb Z_{>0}$. More generally, the integral\begin{align}\begin{split}C_{m,n}={}&\frac{2^{1+2(n-1)[1-(-1)^{m}]}}{\pi^{m+1}}\int_0^\infty \frac{[\pi I_0(t)+i K_{0}(t)]^{m}-[\pi I_0(t)-i K_{0}(t)]^{m}}{i}\times\\&\times[K_0(t)]^{m}(2t)^{2n+m-3}\D t\end{split}\label{eq:Crandall_int}\end{align}  evaluates to a positive integer for each $ m,n\in\mathbb Z_{>0}$. \end{enumerate}\end{theorem}

The next theorem includes two sets of non-linear sum rules, which were originally discovered by  Broadhurst--Mellit \cite[(6.12) and (7.13)]{BroadhurstMellit2016} through numerical experiments on moderately-sized determinants.
An analytic proof has recently been found \cite{Zhou2017BMdet} for Broadhurst--Mellit determinants that come in arbitrary sizes.\begin{theorem}[Broadhurst--Mellit determinant formulae]\label{thm:BMdet}Define  $\mathbf M_k$  and $\mathbf N_k $ as  $k\times k$
matrices with elements
\begin{align}
(\mathbf M_k)_{a,b}:={}&\int_0^\infty[I_0(t)]^a[K_0(t)]^{2k+1-a}t^{2b-1}\D t,
\label{Mk}\\(\mathbf N_k)_{a,b}:={}&\int_0^\infty[I_0(t)]^a[K_0(t)]^{2k+2-a}t^{2b-1}\D t.\label{Nk}\end{align}Then we have the following determinant formulae:\begin{align}
\det\mathbf M_k={}&\prod_{j=1}^k\frac{(2j)^{k-j}\pi^j}{\sqrt{(2j+1)^{2j+1}}},\\\det\mathbf N_k={}&\frac{2\pi^{(k+1)^2/2}}{\varGamma((k+1)/2)}\prod_{j=1}^{k+1}\frac{(2j-1)^{k+1-j}}{(2j)^j},
\end{align}where  Euler's gamma function\footnote{Throughout this survey, we reserve the upright $\varGamma$ for Euler's gamma function, and write $ \Gamma$ in slanted typeface for congruence subgroups (to be introduced in \S\ref{subsec:mod_form}). } is defined by  $ \varGamma(x):=\int_0^\infty t^{x-1}e^{-t}\D t$ for $x>0$.
\end{theorem}\subsection{Some arithmetic properties of Bessel moments\label{subsec:arith_BM}}In what follows, we write  $f_{k,N} $ for a \textit{modular form} (see  \S\ref{subsec:mod_form}  for technical details) of weight $k$ and level $N$, and define its $L $-function through a Mellin transform:  \begin{align}
L(f_{k,N},s):=\frac{(2\pi )^{s}}{\varGamma(s)}\int_0^{\infty} f_{k,N}(iy)y^{s-1}\D y.
\end{align}A special $L$-value $L(f_{k,N},s)$ is said to be \textit{critical}, if $ s\in\mathbb Z\cap(0,k)$. In this survey, we  will be interested in the following three special modular forms:\begin{align}
f_{3,15}(z)={}&[\eta(3z)\eta(5z)]^3+[\eta(z)\eta(15z)]^3,\label{eq:f315_defn}\\f_{4,6}(z)={}&[\eta(z)\eta(2z)\eta(3z)\eta(6z)]^{2},\label{eq:f46_defn}\\f_{6,6}(z)={}&
\frac{ [\eta (2 z) \eta (3 z)]^{9}}{[\eta (z)\eta (6 z)]^{3}}+\frac{ [\eta ( z) \eta (6 z)]^{9}}{[\eta (2z)\eta (3 z)]^{3}},\label{eq:f66_defn}\end{align}where the  Dedekind eta function is defined as $ \eta(z):=e^{\pi iz/12}\prod_{n=1}^\infty(1-e^{2\pi inz})$ for complex numbers $z$ satisfying $ \IM z>0$.
For $ y>0$, one can deduce  \begin{align}
f_{k,N}\left( \frac{i}{Ny} \right)=(\sqrt{N}y)^kf_{k,N}(iy)
\end{align} from the modular transformation  $ \eta(-1/\tau)=\sqrt{\tau/i}\eta(\tau)$ for $ \tau/i>0$. Consequently, the   $ L$-functions attached to these three modular forms satisfy the following reflection formulae \cite[(95), (106), (138)]{Broadhurst2016}:\begin{align}
\varLambda(f_{k,N},s):=\left( \frac{\sqrt{N}}{\pi} \right)^s\varGamma\left( \frac{s}{2} \right)\varGamma\left( \frac{s+1}{2} \right)L(f_{k,N},s)=\varLambda(f_{k,N},k-s).\label{eq:Lambda_refl}
\end{align}

 The studies of the Bessel moments $\IKM(1,4;1) $ and $ \IKM(2,3;1)$ had been initiated by Bailey--Borwein--Broadhurst--Glasser \cite[\S5]{BBBG2008}. Back in 2008, it was analytically confirmed that \begin{align}
\IKM(2,3;1)=\frac{\sqrt{15}\pi}{2}C\label{eq:IKM231}
\end{align}where  \begin{align} C=\frac{1}{240 \sqrt{5}\pi^{2}}\varGamma \left(\frac{1}{15}\right) \varGamma \left(\frac{2}{15}\right) \varGamma \left(\frac{4}{15}\right) \varGamma \left(\frac{8}{15}\right)\label{eq:BolognaC}\end{align} is the \textit{Bologna constant} attributed to Broadhurst \cite{Broadhurst2007,BBBG2008} and Laporta \cite{Laporta2008}. Later on, it was realized that \eqref{eq:IKM231} can be rewritten as $ \IKM(2,3;1)=\frac34L(f_{3,15},2)=\frac{3\pi}{2\sqrt{15}}L(f_{3,15},1)$ \cite[(96)--(97)]{Broadhurst2016}, thanks to the work of Rogers--Wan--Zucker \cite[Theorem 5]{RogersWanZucker2015}.
An innocent-looking conjecture  $ \IKM(1,4;1)=\frac{2\pi}{\sqrt{15}}\IKM(2,3;1)$ was proposed in 2008 \cite[(95)]{BBBG2008}, but was not resolved until Bloch--Kerr--Vanhove carried out a  \textit{tour de force}      in motivic cohomology during 2015 \cite{BlochKerrVanhove2015}, and Samart elucidated the computations of special gamma values in 2016 \cite{Samart2016}. We have recently simplified \cite[Theorem 2.2.2]{Zhou2017WEF} the result of Bloch--Kerr--Vanhove and Samart, as stated in the theorem below.  \begin{theorem}[3-loop sunrise via Bologna constant]\label{thm:Bologna}We have \begin{align}
\IKM(1,4;1)=\pi^{2}C=\frac{\pi^{2}}{5}L(f_{3,15},1)=\frac{3\pi}{2\sqrt{15}}L(f_{3,15},2).
\end{align}\end{theorem}

 Based on a discussion with Francis Brown at Les Houches in 2010, and encouraged by a result of Zhiwei Yun published in 2015 \cite{Yun2015},
David  Broadhurst discovered some relations between $ \IKM(a,6-a;1)$ and $L( f_{4,6},s)$ \cite[\S7.3]{Broadhurst2016}, as well as between  $ \IKM(a,8-a;1)$ and $L(f_{6,6},s) $ \cite[\S7.6]{Broadhurst2016}. All these conjectures have been verified recently \cite[\S\S4--5]{Zhou2017WEF}, so they are included in the theorem below. \begin{theorem}[Critical $L$-values for 6-Bessel and 8-Bessel problems]\label{thm:specL}\begin{enumerate}[leftmargin=*,  label=\emph{(\alph*)},ref=(\alph*),
widest=a, align=left]\item\label{itm:6Bessel} We have \begin{align}
\frac{3}{\pi^{2}}\IKM(1,5;1)=\IKM(3,3;1)={}&\frac{3}{2}L(f_{4,6},2),\label{eq:IKM151IKM331}\\\IKM(2,4;1)={}&\frac{\pi ^{2}}{2}L(f_{4,6},1)=\frac{3}{2}L(f_{4,6},3),\label{eq:IKM241}
\end{align}where the first equality in \eqref{eq:IKM151IKM331} comes from Theorem \ref{thm:B3G_Crandall}\ref{itm:B3G}, and the last equality in \eqref{eq:IKM241} descends from \eqref{eq:Lambda_refl}.\item\label{itm:8Bessel} We have\begin{align}
\IKM(4,4;1)={}&L(f_{6,6},3),\\\frac{1}{\pi^{2}}\IKM(1,7;1)=\IKM(3,5;1)={}&\frac{9}{4}L(f_{6,6},4),\label{eq:IKM171_351_eta_int}\\\IKM(2,6;1)={}&\frac{27}{4}L(f_{6,6},5),
\label{eq:IKM261}\end{align}where the first equality in \eqref{eq:IKM171_351_eta_int} follows from Theorem \ref{thm:B3G_Crandall}\ref{itm:B3G},\end{enumerate}\end{theorem}\subsection{Plan of proofs\label{subsec:Leitfaden}}To help our readers navigate through this  survey, we present the \textit{Leitfaden} in Table~\ref{tab:Leitfaden}.\begin{table}[hb]\caption{Organizational chart\label{tab:Leitfaden}}
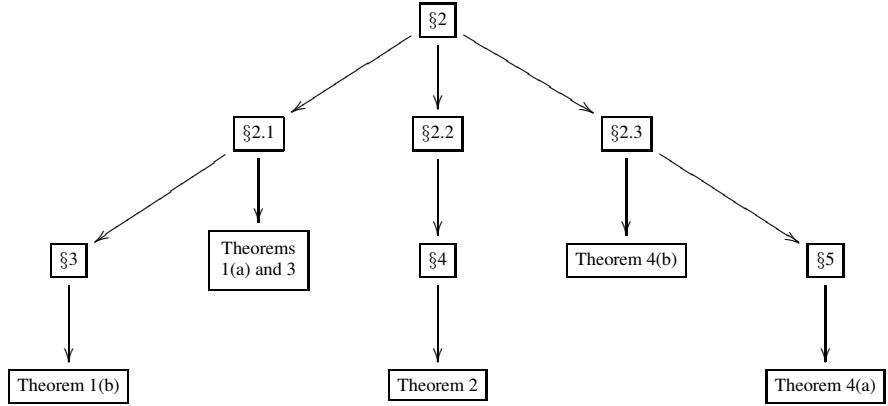
\begin{equation*}
\begin{scriptsize}\xymatrix{ &&\fbox{\S\ref{sec:toolkit}}\ar[dl]\ar[d]\ar[dr]&\\&\fbox{\S\ref{subsec:Wick}}\ar[dl]\ar[d]&\fbox{\S\ref{subsec:VanhoveL}}\ar[d]&\fbox{\S\ref{subsec:mod_form}}\ar[d]\ar[dr]\\\fbox{\S\ref{sec:Crandall}}\ar[d]&\fbox{\begin{tabular}{c}Theorems\\ \ref{thm:B3G_Crandall}\ref{itm:B3G} and \ref{thm:Bologna}\end{tabular}}&\fbox{\S\ref{sec:BMdet}}\ar[d]&\fbox{Theorem \ref{thm:specL}\ref{itm:8Bessel}}&\fbox{\S\ref{sec:L-value}}\ar[d]\\\fbox{Theorem \ref{thm:B3G_Crandall}\ref{itm:Crandall}}&&\fbox{Theorem \ref{thm:BMdet}}&&\fbox{Theorem \ref{thm:specL}\ref{itm:6Bessel}}}\end{scriptsize}
\end{equation*}\end{table}

In \S\ref{subsec:Wick}, we begin with a summary of useful analytic properties for Bessel functions, which result in a proof of Theorem \ref{thm:B3G_Crandall}\ref{itm:B3G}. We then present
Wick rotations, which are special contour deformations connecting moment problems for $\IKM(a,b;n)$  to those for     \begin{align} \JYM(\alpha,\beta;n):=\int_0^\infty[J_0(t)]^\alpha[Y_0(t)]^{\beta}t^{n}\D t,\end{align}where $ a+b=\alpha+\beta$, and  $ J_0(x):=\frac{2}{\pi}\int_0^{\pi/2}\cos(x\cos\varphi)\D\varphi$,  $ Y_0(x):=-\frac{2}{\pi}\int_0^\infty\cos(x\cosh u)\D u$ are Bessel functions of the zeroth order, defined for $x>0$. The     $\JYM$ problems have some desirable properties \cite{BSWZ2012,Zhou2017Int4Pnu}, which lead us to a quick proof of Theorem \ref{thm:Bologna}. Further applications of Wick rotations are given in \S\ref{sec:Crandall}, in the context of Theorem \ref{thm:B3G_Crandall}.

In \S\ref{subsec:VanhoveL}, we give a brief overview of Vanhove's differential equations \cite[\S9]{Vanhove2014Survey}, and compute certain Wro\'nskian determinants involving Bessel moments. These preparations allow us to present the main ideas behind the proof of Theorem \ref{thm:BMdet}, in \S\ref{sec:BMdet}.

In \S\ref{subsec:mod_form}, we describe how to obtain critical $L$-values via integrations over products of certain modular forms, illustrating our general procedures with the proof of  Theorem \ref{thm:specL}\ref{itm:8Bessel}. Some extensions in \S\ref{sec:L-value} then lead to a sketched proof of all the statements in Theorem \ref{thm:specL}.

In \S\ref{sec:open}, we wrap up this survey with some open questions on Bessel moments.
\section{Toolkit\label{sec:toolkit}}\subsection{Wick rotations of Bessel moments\label{subsec:Wick}}As we may recall, for $ \nu \in\mathbb C,-\pi<\arg z<\pi$,  the Bessel functions $ J_\nu$ and $Y_\nu$ are defined by\begin{align}
J_\nu(z):={}&\sum_{k=0}^\infty\frac{(-1)^k}{k!\varGamma(k+\nu+1)}\left( \frac{z}{2} \right)^{2k+\nu},&Y_\nu(z):={}&\lim_{\mu\to\nu}\frac{J_\mu(z)\cos(\mu\pi )-J_{-\mu}(z)}{\sin(\mu\pi)},\label{eq:JY_series}\intertext{which may be compared to the modified Bessel functions
$I_\nu$ and $K_\nu$:}I_\nu(z):={}&\sum_{k=0}^\infty\frac{1}{k!\varGamma(k+\nu+1)}\left( \frac{z}{2} \right)^{2k+\nu},&K_\nu(z):={}&\frac{\pi}{2}\lim_{\mu\to\nu}\frac{I_{-\mu}(z)-I_\mu(z)}{\sin(\mu\pi)}.\label{eq:IK_series}\end{align}Here, the fractional powers of complex numbers are defined through $ w^\beta=\exp(\beta\log w)$ for $\log w=\log|w|+i\arg w$, where $ |\arg w|<\pi$.

The  cylindrical Hankel functions $H^{(1)}_0(z)=J_0(z)+iY_0(z)$ and  $H^{(2)}_0(z)=J_0(z)-iY_0(z)$    are both  well defined for  $ -\pi<\arg z<\pi$. In view of \eqref{eq:JY_series} and \eqref{eq:IK_series}, we can verify  \begin{align} J_0(ix)=I_0(x)\quad\text{ and}\quad\frac{\pi i}{2}H_0^{(1)}(ix)=K_0(x)\label{eq:JH_Imz} \end{align} along with\begin{align}
H_0^{(1)}(\pm x+i0^+)=\pm J_{0}(x)+i Y_0(x) \label{eq:H1_x}
\end{align}  for $ x>0$.

As $ |z|\to\infty,-\pi<\arg z<\pi$, we have the following asymptotic expansions \cite[\S7.2]{Watson1944Bessel}:  \begin{align}\left\{\begin{array}{r@{\;=\;}l}\smash[t]{H_0^{(1)}(z)}&\smash[t]{\displaystyle\sqrt{\frac{2}{\pi z}}e^{i\left(z-\frac{\pi}{4}\right)}\left\{1+\sum _{n=1}^{N} \frac{\left[\varGamma \left(n+\frac{1}{2}\right)\right]^2}{(2 i z)^n \pi  n!}+O\left( \frac{1}{|z|^{N+1}} \right)\right\}},\\H_0^{(2)}(z)&\smash[b]{\displaystyle \sqrt{\frac{2}{\pi z}}e^{i\left(\frac{\pi}{4}-z\right)}\left\{1+\sum _{n=1}^{N} \frac{\left[\varGamma \left(n+\frac{1}{2}\right)\right]^2}{(-2 i z)^n \pi  n!}+O\left( \frac{1}{|z|^{N+1}} \right)\right\}},\end{array}\right.
\label{eq:H0_asympt}\end{align}which allow us to establish a vanishing identity\begin{align}
\int_{-i\infty}^{i\infty}[H_0^{(1)}(z)H_0^{(2)}(z)]^mz^{n}\D z=0,\quad n\in\mathbb Z\cap[0,m-1)\label{eq:HHM}
\end{align}by closing the contour rightwards. One can transcribe the last vanishing integral into the statements in Theorem \ref{thm:B3G_Crandall}\ref{itm:B3G}, bearing in mind that \begin{align}
H_0^{(1)}(it)H_0^{(2)}(it)=\frac{4K_{0}(|t|)}{\pi^{2}}\left[ K_0(|t|)-\frac{\pi it}{|t|}I_0(|t|) \right],\quad \forall t\in(-\infty,0)\cup(0,\infty).\label{eq:H1H2Imz}
\end{align} \begin{figure}[b]
\begin{minipage}{0.3\textwidth}\begin{tikzpicture}
\draw[->](-0.25,0)--(3,0)node[right]{$\RE z$};
\draw[->](0,-0.25)--(0,3)node[above]{$\IM z$};
\draw[<-,ultra thick] (1.25,0)--(2.75,0);
\draw[ultra thick] (0,0)--(1.75,0);
\draw[->,ultra thick] (0,0)--(0,1.5);
\draw[ultra thick] (0,0.75)--(0,2.75);
\draw[->,ultra thick] plot[samples=100,domain=90:45] ({2.72*cos(\x)},{2.72*sin(\x)});
\draw[ultra thick] plot[samples=100,domain=55:0] ({2.72*cos(\x)},{2.72*sin(\x)});

\draw(1.35,-0.25) node[below]{\textbf{(a)}};
\end{tikzpicture}
\end{minipage}\hspace{5em}\begin{minipage}{0.5\textwidth}
\begin{tikzpicture}
\draw(0,-0.25) node[below]{\textbf{(b)}};
\draw[->](-3,0)--(3,0)node[right]{$\RE z$};
\draw[->](0,-0.25)--(0,3)node[above]{$\IM z$};
\draw[->,ultra thick] (-2.75,0.1)--(0,0.1);
\draw[ultra thick] (-1,0.1)--(2.75,0.1);
\draw[ultra thick] plot[samples=100,domain=178:80] ({2.72*cos(\x)},{2.72*sin(\x)});
\draw[<-,ultra thick] plot[samples=100,domain=90:2] ({2.72*cos(\x)},{2.72*sin(\x)});
\end{tikzpicture}\end{minipage}
\caption{\textbf{a} Wick rotation that turns an $\IKM$ to a sum of several $\JYM$'s. Note that the contribution from the circular arc tends to zero as  $ |z|\to\infty$, thanks to Jordan's lemma being applicable to the asymptotic behavior of Hankel functions.~\textbf{b} Contour of integration that leads to a cancelation formula for $\JYM$'s. }
\label{fig:Wick}       
\end{figure}
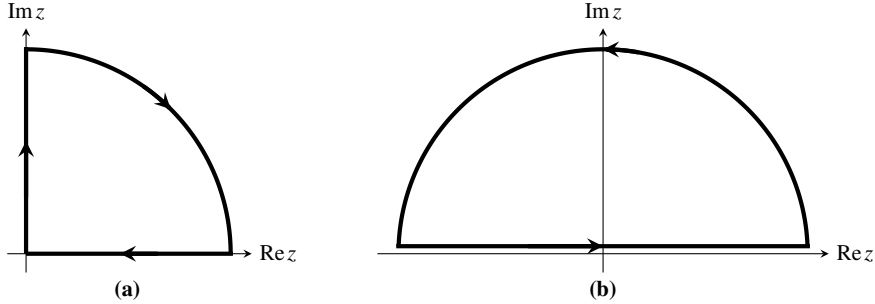

\begin{lemma}[An application of Wick rotation]We have the following relation between $\IKM$ and $\JYM$: \begin{align}
\left(\frac2\pi\right)^4\IKM(1,4;1)=-\JYM(5,0;1)+6\JYM(3,2;1)-\JYM(1,4;1).\label{eq:IKMtoJYM}
\end{align}\end{lemma}\begin{proof}From \eqref{eq:JH_Imz}, we know that \begin{align}
\left(\frac2\pi\right)^4\IKM(1,4;1)=-\RE\int_{0}^{i\infty} J_0^{\vphantom{(1)}}(z)[H_0^{(1)}(z)]^4 z\D z,\label{eq:IKM141Imz}
\end{align}where the contour runs along the positive $\IM z$-axis.

Noting that  the asymptotic behavior of $ J_0(z)=[H_0^{(1)}(z)+H_0^{(2)}(z)]/2$ can be inferred from \eqref{eq:H0_asympt},   we can rotate the contour 90$ ^\circ$ clockwise, from the positive $ \IM z$-axis to the positive $\RE z$-axis (see Fig.~\ref{fig:Wick}\textit{a}), thereby equating \eqref{eq:IKM141Imz} with \begin{align}
-\RE\int_{0}^{\infty} J_0^{\vphantom{(1)}}(x)[H_0^{(1)}(x)]^4 x\D x=-\RE\int_{0}^{\infty} J_0(x)[J_0(x)+iY_0(x)]^4 x\D x,
\end{align}     hence the right-hand side of \eqref{eq:IKMtoJYM}.
\qed\end{proof}
\begin{proposition}[Evaluation of $\IKM(1,4;1)$]We have\begin{align}
\IKM(1,4;1)=\frac{\pi^{4}}{30}\JYM(5,0;1)=\pi^2C,\label{eq:IKM141JYM501}
\end{align}where $C$ is the Bologna constant defined in \eqref{eq:BolognaC}.\end{proposition}\begin{proof}For $ \ell,m,n\in\mathbb Z_{\geq0}$ satisfying either $\ell-(m+n)/2<0; m<n$ or $ \ell-m=\ell-n<-1$, we can prove \begin{align}
\int_{i0^+-\infty}^{i0^++\infty}[J_0^{\vphantom{(1)}}(z)]^m[H_0^{(1)}(z)]^nz^\ell\D z:=\lim_{\varepsilon\to0^+}\lim_{R\to\infty}\int_{i\varepsilon-R}^{i\varepsilon+R}[J_0^{\vphantom{(1)}}(z)]^m[H_0^{(1)}(z)]^nz^\ell\D z=0,\label{eq:BHJ}
\end{align}by considering the contour in Fig.~\ref{fig:Wick}\textit{b}. According to \eqref{eq:H1_x} and $J_0(-x)=J_0(x) $, we may reformulate \eqref{eq:BHJ} as\begin{align}
\int_0^\infty [J_{0}(x)]^m\left\{ [J_{0}(x)+i Y_0(x)]^{n}+(-1)^\ell [-J_{0}(x)+i Y_0(x)]^{n}\right\}x^\ell\D x=0,\label{eq:JY_BHJ}\tag{\ref{eq:BHJ}$'$}
\end{align} which is a convenient cancelation formula for $\JYM$'s.

With  $
J(J^4-6 J^2 Y^2+ Y^4)-\frac{2J^2}{3}  [(J+i Y)^3-(-J+i Y)^3]-\frac{(J+i Y)^5-(-J+i Y)^5}{10} =-\frac{8 J^5}{15}$ in hand, we can identify  the right-hand side of \eqref{eq:IKMtoJYM} with $ \frac8{15}\JYM(5,0;1)$. This proves the first equality in \eqref{eq:IKM141JYM501}. The second equality can be directly deduced from \cite[(5.2)]{BSWZ2012}. \qed\end{proof}

So far, we have recapitulated an analytic proof of  Theorem \ref{thm:Bologna}, as originally given in \cite[\S2]{Zhou2017WEF}. It is worth pointing out that Kluyver's  function \cite{Kluyver1906}\begin{align}
p_{n}(x)=\int_0^\infty J_0(xt)[J_0(t)]^n xt\D t
\end{align}characterizes the probability density of the distance $x$ traveled by  a rambler, who walks in the Euclidean plane, taking $n$ consecutive steps of unit lengths, aiming at uniformly distributed random directions.
The analytic properties of such probability densities have been extensively studied \cite{BNSW2011,BSWZ2012,BSW2013,BSV2016,Zhou2017PlanarWalks}. Recently, we have  shown \cite[Theorem 5.1]{Zhou2017PlanarWalks} that $ p_n(x)$ is expressible through Feynman diagrams when $n$ is odd, as stated in the theorem below.\begin{theorem}[$ p_{2j+1}(x)$ as Feynman diagrams]\label{thm:p_odd}For each $ j\in\mathbb Z_{>1}$, the function $p_{2j+1}(x),0\leq x\leq 1 $ is a unique $ \mathbb Q$-linear combination of \begin{align}
\int_0^\infty I_0(xt)[I_0(t)]^{2m+1}\left[\frac{K_{0}(t)}{\pi}\right]^{2(j-m)}xt\D t,\quad\text{where } m\in\mathbb Z\cap\left[0,\frac{j-1}{2}\right].\label{eq:pn_IKM}
\end{align} (When $ j=1$, the same is true for  $ 0\leq x<1$.)\end{theorem}\subsection{Vanhove's differential equations and Wro\'nskians of Bessel moments\label{subsec:VanhoveL}}In \cite[\S9]{Vanhove2014Survey}, Vanhove has constructed $ n$-th order differential operators $ \widetilde L_n$ (written in the variable $u$ in this survey) so that the relation \begin{align}
\widetilde L_n\int_0^\infty I_0(\sqrt{u}t)[K_0(t)]^{n+1}t\D t=\mathrm{const}
\end{align} holds for all $ n\in\mathbb Z_{>0}$ and $u\in(0,(n+1)^2)$. The first few Vanhove operators  $ \widetilde L_n$ are listed in Table \ref{tab:VanhoveL}, where $ D^n=\partial ^n/\partial  u^n$ for $n\in\mathbb Z_{>0} $ and $ D^0$ is the identity operator.

\begin{table}[ht]\caption{The first few Vanhove differential operators (abridged from \cite[\S9, Table 1]{Vanhove2014Survey})\label{tab:VanhoveL}}\begin{tabular}{p{0.5cm}p{11cm}}
\hline\noalign{\smallskip}
$n$ & $ \widetilde L_n$  \\
\noalign{\smallskip}\svhline\noalign{\smallskip}
1 &$
    u(u-4)D^{1}+(u-2 )D^{0}$\\2&$ u(u- 1) (u-
9)D^{2}+(
  3\,{u}^{2}-20u+9 )D^{1}+( u-3 ) D^0  $\\3&$u^2 (u - 4) (u - 16)D^{3}+6 u (u^2 - 15 u + 32)D^{2}+(7 u^2 - 68 u + 64)D^{1}+(u-4)D^{0}$\\4&$u^2 (u-1) (u-9)(u-25) D^{4}+2 u (5 u^3-140 u^2+777 u-450)D^{3}+(25 u^3-518 u^2+1839 u-450)D^{2}+(3 u-5) (5 u-57)D^{1}+(u-5)D^{0}$\\
\noalign{\smallskip}\hline\noalign{\smallskip}
\end{tabular}
\end{table}

In general, for each $ n\in\mathbb Z_{\geq1}$, Vanhove's operator $ \widetilde L_n$  satisfies\begin{align}
\left\{ \begin{array}{c}
t\widetilde L_nI_0(\sqrt{u}t)=\frac{(-1)^n}{2^n}L^*_{n+2}\frac{I_0(\sqrt{u}t)}{t}, \\
t\widetilde L_nK_0(\sqrt{u}t)=\frac{(-1)^n}{2^n}L^*_{n+2}\frac{K_0(\sqrt{u}t)}{t}, \\
\end{array} \right.
\end{align}where $ L_{n+2}^*$ is the formal adjoint to the Borwein--Salvy operator $ L_{n+2}$ \cite{BorweinSalvy2007}, the latter of which is the $(n+1)$-st symmetric power of the Bessel differential operator $ (t\partial /\partial t)^2-t^2$ that annihilates both $ I_0(t)$ and $K_0(t)$. Using the Bronstein--Mulders--Weil algorithm \cite{BMW1997} for symmetric powers, we have shown \cite[Lemma 4.2]{Zhou2017BMdet} that the following homogeneous differential equations {\allowdisplaybreaks\begin{align}
\widetilde L_n\left[\int _0^\infty I_0(\sqrt{u}t)[K_0(t)]^{n+1}t\D t+  (n+1)\int_0^\infty K_0(\sqrt{u}t)I_{0}(t)[K_0(t)]^{n}t\D t\right]=0,\\\widetilde L_n\int_0^\infty I_0(\sqrt{u}t)[I_0(t)]^{j-1}[K_0(t)]^{n+2-j}t\D t=0,\quad \forall j\in\mathbb Z\cap\left[2,\frac{n}{2}+1\right],\\\widetilde L_n\int_0^\infty K_0(\sqrt{u}t)[I_0(t)]^{j}[K_0(t)]^{n+1-j}t\D t=0,\quad \forall j\in\mathbb Z\cap\left[2,\frac{n+1}{2}\right]
\end{align}}hold for $ u\in(0,1)$.

For $N\in\mathbb Z_{>1}$, we write  $ W[f_1(u),\dots,f_N(u)]$ for the Wro\'nskian determinant $ \det(D^{i-1}f_j(u))_{1\leq i,j\leq N}$. In \cite[\S4.1]{Zhou2017BMdet}, we have constructed
some Wro\'nskians as precursors to Broadhurst--Mellit determinants $\det\mathbf M_k $ and $ \det\mathbf N_k$ (see Theorem \ref{thm:BMdet}).
Concretely speaking, for each $ k\in\mathbb Z_{\geq2}$, we set \begin{align}
\left\{\begin{array}{l}\mu^\ell_{k,1}(u)=\frac{1}{2k+1}\int _0^\infty \{I_0(\sqrt{u}t)K_0(t)+2k K_0(\sqrt{u}t)I_{0}(t)\}[K_0(t)]^{2k-1}t^{2\ell-1}\D t,\\\mu^\ell_{k,j}(u)=\int_0^\infty I_0(\sqrt{u}t)[I_0(t)]^{j-1}[K_0(t)]^{2k+1-j}t^{2\ell-1}\D t,\forall j\in\mathbb Z\cap[2,k],\\\mu^\ell_{k,j}(u)=\int_0^\infty K_0(\sqrt{u}t)[I_0(t)]^{j-k+1}[K_0(t)]^{3k-1-j}t^{2\ell-1}\D t,\forall j\in\mathbb Z\cap[k+1,2k-1],\end{array}\right.
\label{eq:mu_defn}\end{align}and \begin{align}
\left\{\begin{array}{l}\nu^\ell_{k,1}(u)=\frac{1}{2(k+1)}\int _0^\infty \{I_0(\sqrt{u}t)K_0(t)+(2k+1) K_0(\sqrt{u}t)I_{0}(t)\}[K_0(t)]^{2k}t^{2\ell-1}\D t,\\\nu^\ell_{k,j}(u)=\int_0^\infty I_0(\sqrt{u}t)[I_0(t)]^{j-1}[K_0(t)]^{2k+2-j}t^{2\ell-1}\D t,\forall j\in\mathbb Z\cap[2,k+1],\\\nu^\ell_{k,j}(u)=\int_0^\infty K_0(\sqrt{u}t)[I_0(t)]^{j-k}[K_0(t)]^{3k+1-j}t^{2\ell-1}\D t,\forall j\in\mathbb Z\cap[k+2,2k],\end{array}\right.
\end{align}and consider the  Wro\'nskian determinants $\Omega_{2k-1}(u):=W[\mu^1_{k,1}(u),\dots,\mu^1_{k,2k-1}(u)] $, $\omega_{2k}(u):=W[\nu^1_{k,1}(u),\dots,\nu^1_{k,2k}(u)] $.
For $ k\in\mathbb Z_{\geq2}$, Vanhove's operators $ \widetilde L_{2k-1}$ and $\widetilde L_{2k} $ take the following forms  \cite[(9.11)--(9.12)]{Vanhove2014Survey}:\begin{align}
\widetilde L_{2k-1}={}&\mathfrak{m}_{2k-1}(u)D^{2k-1}+\frac{2k-1}{2}\frac{\D\mathfrak{m}_{2k-1}(u)}{\D u}D^{2k-2}+L.O.T.,\\\widetilde L_{2k}={}&\mathfrak n_{2k}(u)D^{2k}+k\frac{\D\mathfrak{n}_{2k}(u)}{\D u}D^{2k-1}+L.O.T.,
\end{align}where\begin{align}
\mathfrak{m}_{2k-1}(u)=u^k \prod _{j=1}^k[u-(2j)^2],\quad \mathfrak n_{2k}(u)=u^k \prod _{j=1}^{k+1}[u-(2j-1)^2],\label{eq:m_n_poly}
\end{align}and ``$ L.O.T.$'' stands for ``lower order terms''.
 Therefore, we have the following evolution equations for Wro\'nskians:\begin{align}
D^{1}\Omega_{2k-1}(u)={}&\frac{2k-1}{2}\Omega_{2k-1}(u)D^1\log\frac{1}{u^k \prod _{j=1}^k[(2j)^2-u]},\label{eq:Omega_evolv}\\D^{1}\omega_{2k}(u)={}&k\omega_{2k}(u)D^1\log\frac{1}{u^k \prod _{j=1}^{k+1}[(2j-1)^2-u]},
\end{align}where $0<u<1$. These differential equations will play crucial r\^oles in the proof of Theorem \ref{thm:BMdet} in \S\ref{sec:BMdet}.

\subsection{Modular forms and their integrations\label{subsec:mod_form}}Let \begin{align}
\Gamma_0(N):=\left\{ \left.\begin{pmatrix}a & b \\
c & d \\
\end{pmatrix}\right|a,b,c,d\in\mathbb Z;ad-bc=1;c \equiv0\,(\!\bmod N)\right\}
\end{align}be the  Hecke congruence group of level $N\in\mathbb Z_{>0}$. For a given Dirichlet character $ \chi$, a modular form $ M_{k,N}(z)$ of weight $k$, level $N$ and multiplier $\chi$ is a holomorphic\footnote{For technical requirements on  holomorphy at $ i\infty$ (more precisely, the $ \Gamma_0(N)$ images of $i\infty$, namely, the cusps $ \Gamma_0(N)\backslash\mathbb P^1(\mathbb Q)=\Gamma_0(N)\backslash(\mathbb Q\cup\{i\infty\})$), see \cite[Definition 1.2.3]{DiamondShurman}.} function that transforms like\footnote{For the modular forms  $ f_{4,6}(z)$ in \eqref{eq:f46_defn} and $ f_{6,6}(z)$ in \eqref{eq:f66_defn}, the multiplier is the trivial Dirichlet character $ \chi(d)\equiv 1$. For the modular form $ f_{3,15}(z)$ in \eqref{eq:f315_defn}, we have $ \chi(d)=\left(\frac{d}{15}\right)$  \cite[Proposition 5.1]{PetersTopVlugt1992}, where the Dirichlet character is defined through a Jacobi--Kronecker symbol.} \begin{align}
M_{k,N}\left( \frac{az+b}{cz+d} \right)=(cz+d)^k\chi(d)M_{k,N}(z),
\end{align}where $ \left(\begin{smallmatrix}a&b\\c&d\end{smallmatrix}\right)$ runs over all the members of $ \Gamma_0(N)$, and $z$ is an arbitrary point in the upper half-plane $\mathfrak H:=\{w\in\mathbb C|\IM w>0\}$. Modular forms of weight $0$\ (relaxing the requirement on holomorphy at cusps) are called modular functions. These $ \Gamma_0(N)$-invariant modular functions are effectively defined on the moduli space  $ Y_0(N)(\mathbb C)=\Gamma_0(N)\backslash\mathfrak H$ (see Fig.~\ref{fig:fun_domains}) for isomorphism classes of complex elliptic curves.

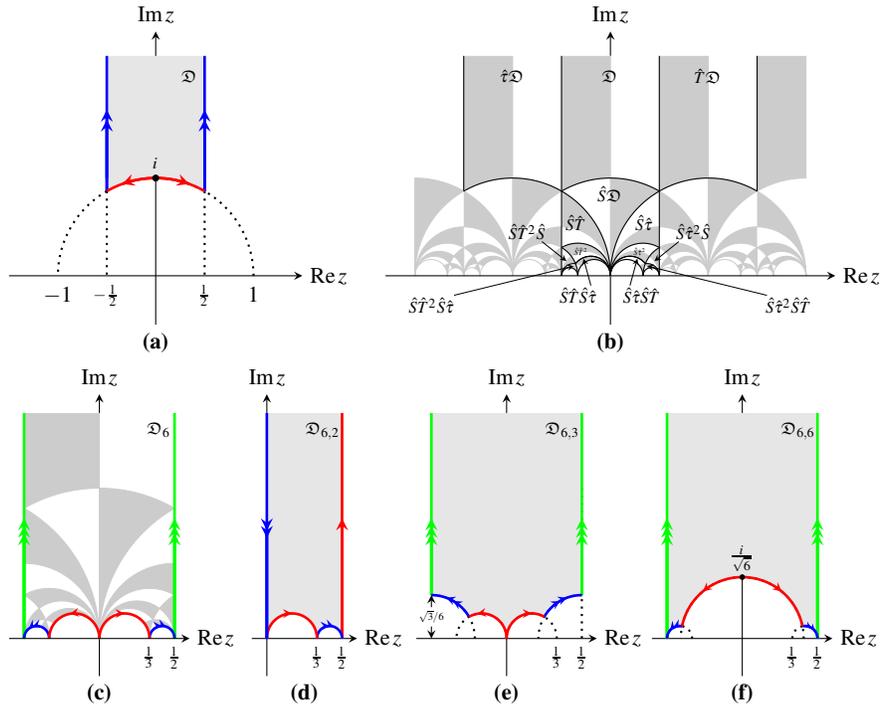
\begin{figure}[b]
\begin{minipage}{0.3\textwidth}\begin{tikzpicture}[scale=1.3]

\draw[->](0,-.5)--(0,2.5)node[above]{$\IM z$};
\draw[color=gray!20!white,fill] (.5,2.25)--(.5,{sqrt(3)/2})--plot[samples=120,domain=60:120] ({cos(\x)},{sin(\x)})--(-.5,{sqrt(3)/2})--(-.5,2.25);

\draw[->](-1.5,0)--(1.5,0)node[right]{$\RE z$};
\draw[dotted,thick](-0.5,{sqrt(3)/2})--(-0.5,0)node[below]{$^{-\frac12} $};
\draw[dotted,thick](0.5,{sqrt(3)/2})--(0.5,0)node[below]{$^{\frac12} $};
\draw[dotted, thick] plot[samples=180,domain=0:90] ({cos(\x)},{sin(\x)});

\draw[->,line width=1pt,color=blue]  ({.5},{sqrt(3)/2})--(.5,1.7);
\draw[->,line width=1pt,color=blue]  ({.5},{sqrt(3)/2})--(.5,1.6);
\draw[line width=1pt,color=blue] (.5,1)--(0.5,2.25);
\draw[->,line width=1pt,color=red] plot[samples=40,domain=90:70] ({cos(\x)},{sin(\x)});
\draw[line width=1pt,color=red] plot[samples=60,domain=90:60] ({cos(\x)},{sin(\x)});

\draw[->,line width=1pt,color=blue]  ({-.5},{sqrt(3)/2})--(-.5,1.7);
\draw[->,line width=1pt,color=blue]  ({-.5},{sqrt(3)/2})--(-.5,1.6);
\draw[line width=1pt,color=blue] (-.5,1)--(-0.5,2.25);
\draw[line width=1pt,->,color=red] plot[samples=40,domain=90:110] ({cos(\x)},{sin(\x)});
\draw[line width=1pt,color=red] plot[samples=60,domain=90:120] ({cos(\x)},{sin(\x)});

\draw[fill=black](0,1)circle[radius=0.03]node at (0,1.15){$_i$};
\draw[dotted, thick] plot[samples=120,domain=120:180] ({cos(\x)},{sin(\x)});
\draw(-1,-.03)node[below]{${-1} $};\draw(1,-.03)node[below]{${1} $};
\draw(0,-0.5) node[below]{\textbf{(a)}};
\node at (.35,2) {$^{ \mathfrak D}$};
\end{tikzpicture}
\end{minipage}\hspace{5em}\begin{minipage}{0.5\textwidth}
\begin{tikzpicture}[scale=1.3]
\draw[->](0,2.25)--(0,2.5)node[above]{$\IM z$};
\draw(0,-.5)--(0,0);
\foreach \y in {-2,-1,0,1} \draw[ultra thin,color=gray!40!white] plot[samples=360,domain=0:180] ({\y+5/12+cos(\x)/12},{sin(\x)/12});
\foreach \y in {-1,0,1,2} \draw[ultra thin,color=gray!40!white] plot[samples=360,domain=0:180] ({\y-5/12+cos(\x)/12},{sin(\x)/12});
\foreach \y in {-2,-1,0,1} \draw[ultra thin,color=gray!40!white] plot[samples=360,domain=0:180] ({\y+1/6+cos(\x)/6},{sin(\x)/6});
\foreach \y in {-1,0,1,2} \draw[ultra thin,color=gray!40!white] plot[samples=360,domain=0:180] ({\y-1/6+cos(\x)/6},{sin(\x)/6});

\foreach \y in {-2,-1,0,1} \draw[color=gray!40!white,fill,ultra thin] plot[samples=60,domain=60:90] ({\y+cos(\x)},{sin(\x)})--({\y},0)--plot[samples=120,domain=180:120] ({\y+1+cos(\x)},{sin(\x)});

\foreach \y in {-2,-1,0,1} \draw[color=gray!40!white,fill,ultra thin] plot[samples=120,domain=60:0] ({\y+cos(\x)},{sin(\x)})--plot[samples=180,domain=0:90] ({\y+1/2+cos(\x)/2},{sin(\x)/2});

\foreach \y in {-2,-1,0,1} \draw[color=gray!40!white,fill,ultra thin] plot[samples=180,domain=90:180] ({\y+1/2+cos(\x)/2},{sin(\x)/2})--plot[samples=240,domain=180:60] ({\y+1/3+cos(\x)/3},{sin(\x)/3});

\foreach \y in {-2,-1,0,1} \draw[color=gray!40!white,fill,ultra thin] plot[samples=120,domain=60:36.9] ({\y+1/3+cos(\x)/3},{sin(\x)/3})--plot[samples=360,domain=127:0] ({\y+3/4+cos(\x)/4},{sin(\x)/4})--plot[samples=240,domain=0:120] ({\y+2/3+cos(\x)/3},{sin(\x)/3});

\foreach \y in {-2,-1,0,1} \draw[color=gray!40!white,fill,ultra thin] plot[samples=120,domain=120:143.1] ({\y+2/3+cos(\x)/3},{sin(\x)/3})--plot[samples=360,domain=53:0] ({\y+1/4+cos(\x)/4},{sin(\x)/4})--({\y+1/2},{1/2/sqrt(3)});

\foreach \y in {-2,-1,0,1} \draw[color=gray!40!white,fill,ultra thin] plot[samples=360,domain=180:53] ({\y+1/4+cos(\x)/4},{sin(\x)/4})--plot[samples=120,domain=143.1:158.2] ({\y+2/3+cos(\x)/3},{sin(\x)/3})--plot[samples=360,domain=38.2:180] ({\y+1/5+cos(\x)/5},{sin(\x)/5});

\foreach \y in {-2,-1,0,1} \draw[color=gray!40!white,fill,ultra thin] plot[samples=360,domain=0:141.8] ({\y+4/5+cos(\x)/5},{sin(\x)/5})--plot[samples=360,domain=81.8:53.1] ({\y+5/8+cos(\x)/8},{sin(\x)/8})--plot[samples=360,domain=143.1:0] ({\y+5/6+cos(\x)/6},{sin(\x)/6});

\foreach \y in {-2,-1,0,1} \draw[color=gray!40!white,fill,ultra thin] plot[samples=360,domain=81.8:180] ({\y+5/8+cos(\x)/8},{sin(\x)/8})--plot[samples=360,domain=180:127] ({\y+3/4+cos(\x)/4},{sin(\x)/4})--plot[samples=120,domain=36.9:21.8] ({\y+1/3+cos(\x)/3},{sin(\x)/3});

\foreach \y in {-2,-1,0,1} \draw[color=gray!40!white,fill,ultra thin] plot[samples=120,domain=21.8:0] ({\y+1/3+cos(\x)/3},{sin(\x)/3})--plot[samples=120,domain=0:67.4] ({\y+7/12+cos(\x)/12},{sin(\x)/12})--plot[samples=360,domain=157.4:141.8] ({\y+4/5+cos(\x)/5},{sin(\x)/5});

\foreach \y in {-2,-1,0,1} \draw[color=gray!40!white,fill,ultra thin] plot[samples=120,domain=0:98.2] ({\y+3/8+cos(\x)/8},{sin(\x)/8})--plot[samples=360,domain=38.2:22.6] ({\y+1/5+cos(\x)/5},{sin(\x)/5})--plot[samples=120,domain=112.6:0] ({\y+5/12+cos(\x)/12},{sin(\x)/12});

\foreach \y in {-2,-1,0,1} \draw[color=gray!40!white,fill,ultra thin] plot[samples=120,domain=98.2:126.9] ({\y+3/8+cos(\x)/8},{sin(\x)/8})--plot[samples=360,domain=36.9:0] ({\y+1/6+cos(\x)/6},{sin(\x)/6})--plot[samples=120,domain=180:158.2] ({\y+2/3+cos(\x)/3},{sin(\x)/3});

\foreach \y in {-1,0,1,2} \draw[color=gray!40!white,fill,ultra thin] ({\y},2.25)--({\y},1)--plot[samples=60,domain=90:120] ({\y+cos(\x)},{sin(\x)})--({\y-.5},{sqrt(3)/2})--({\y-.5},2.25);

\draw(0,-0.5) node[below]{\textbf{(b)}};
\node[color=black] at (0,2) {$^{ \mathfrak D}$};
\node[color=black] at (0,.8) {$^{ \hat S\mathfrak D}$};
\node[color=black] at (-0.35,.5) {$^{\hat S\hat T}$};
\node[color=black] at (0.35,.5) {$^{ \hat S\hat \tau}$};
\node[color=black] at (-0.3,.25) {\fontsize{3}{3.6}\selectfont$\hat S\hat T^2$};
\node[color=black] at (0.3,.25) {\fontsize{3}{3.6}\selectfont$\hat S\hat \tau^2$};
\draw[<-,black,text=black,ultra thin] (-0.45,.15)--(-0.7,.35) node at (-0.85,.45){$_{\hat S\hat T^2\hat S}$};
\draw[<-,black,text=black,ultra thin] (0.45,.15)--(0.7,.35) node at (0.85,.45){$_{\hat S\hat \tau^2\hat S}$};
\draw[<-,black,text=black,ultra thin] (-0.275,.175)--(-0.2,-.1) node at (-0.325,-.2){$_{\hat S\hat T\hat S\hat\tau}$};
\draw[<-,black,text=black,ultra thin] (0.275,.175)--(0.2,-.1) node at (0.325,-.2){$_{\hat S\hat \tau\hat S\hat T}$};
\draw[<-,black,text=black,ultra thin] (-0.375,.1)--(-1.6,-.2) node at (-1.825,-.3){$_{\hat S\hat T^2\hat S\hat\tau}$};
\draw[<-,black,text=black,ultra thin] (0.375,.1)--(1.6,-.2) node at (1.825,-.3){$_{\hat S\hat \tau^2\hat S\hat T}$};
\node[color=black] at (-1,2) {$^{\hat \tau \mathfrak D}$};
\node[color=black] at (1,2) {$^{\hat T\mathfrak D}$};

\draw[color=black,thin](-1.5,{sqrt(3)/2})--(-1.5,2.25);
\draw[color=black,thin](1.5,{sqrt(3)/2})--(1.5,2.25);
\draw[color=black,thin](-.5,0)--(-.5,2.25);
\draw[color=black,thin](.5,0)--(.5,2.25);
\draw[color=black,thin] plot[samples=120,domain=60:120] ({cos(\x)},{sin(\x)});
\draw[color=black,thin] plot[samples=240,domain=0:120] ({cos(\x)-1},{sin(\x)});
\draw[color=black,thin] plot[samples=240,domain=60:180] ({cos(\x)+1},{sin(\x)});
\draw[color=black,thin] plot[samples=120,domain=0:60] ({cos(\x)/3-2/3},{sin(\x)/3});
\draw[color=black,thin] plot[samples=120,domain=120:180] ({cos(\x)/3+2/3},{sin(\x)/3});
\draw[color=black,thin] plot[samples=240,domain=0:120] ({cos(\x)/3-1/3},{sin(\x)/3});
\draw[color=black,thin] plot[samples=240,domain=60:180] ({cos(\x)/3+1/3},{sin(\x)/3});
\draw[color=black,thin] plot[samples=360,domain=0:180] ({cos(\x)/6-1/6},{sin(\x)/6});
\draw[color=black,thin] plot[samples=360,domain=0:180] ({cos(\x)/6+1/6},{sin(\x)/6});
\draw[color=black,thin] plot[samples=360,domain=0:180] ({cos(\x)/12-5/12},{sin(\x)/12});
\draw[color=black,thin] plot[samples=360,domain=0:180] ({cos(\x)/12+5/12},{sin(\x)/12});
\draw[color=black,thin] plot[samples=360,domain=38.2:180] ({1/5+cos(\x)/5},{sin(\x)/5});
\draw[color=black,thin] plot[samples=360,domain=0:141.8] ({-1/5+cos(\x)/5},{sin(\x)/5});
\draw[color=black,thin] plot[samples=360,domain=38.2:180] ({1/5+cos(\x)/5},{sin(\x)/5});
\draw[color=black,thin] plot[samples=360,domain=0:141.8] ({-1/5+cos(\x)/5},{sin(\x)/5});
\draw[color=black,thin] plot[samples=120,domain=0:98.2] ({3/8+cos(\x)/8},{sin(\x)/8});
\draw[color=black,thin] plot[samples=120,domain=81.8:180] ({-3/8+cos(\x)/8},{sin(\x)/8});

\draw[->](-2.3,0)--(2.3,0) node[right]{$\RE z$};
\end{tikzpicture}\end{minipage}

\begin{tikzpicture}[scale=2]
\foreach \y in {0}\draw[->]({\y},1.5)--({\y},1.625)node[above]{$\IM z$};
\draw (0,-.25)--(0,0);
\foreach \y in {0} \draw[color=gray!40!white,fill,ultra thin] plot[samples=60,domain=60:90] ({\y+cos(\x)},{sin(\x)})--({\y},0)--plot[samples=120,domain=180:120] ({\y+1+cos(\x)},{sin(\x)});

\foreach \y in {-1}  \draw[color=gray!40!white,fill,ultra thin] plot[samples=120,domain=60:0] ({\y+cos(\x)},{sin(\x)})--plot[samples=180,domain=0:90] ({\y+1/2+cos(\x)/2},{sin(\x)/2});

\foreach \y in {0} \draw[color=gray!40!white,fill,ultra thin] plot[samples=180,domain=90:180] ({\y+1/2+cos(\x)/2},{sin(\x)/2})--plot[samples=240,domain=180:60] ({\y+1/3+cos(\x)/3},{sin(\x)/3});

\foreach \y in {-1} \draw[color=gray!40!white,fill,ultra thin] plot[samples=120,domain=60:36.9] ({\y+1/3+cos(\x)/3},{sin(\x)/3})--plot[samples=360,domain=127:0] ({\y+3/4+cos(\x)/4},{sin(\x)/4})--plot[samples=240,domain=0:120] ({\y+2/3+cos(\x)/3},{sin(\x)/3});

\foreach \y in {0} \draw[color=gray!40!white,fill,ultra thin] plot[samples=120,domain=120:143.1] ({\y+2/3+cos(\x)/3},{sin(\x)/3})--plot[samples=360,domain=53:0] ({\y+1/4+cos(\x)/4},{sin(\x)/4})--({\y+1/2},{1/2/sqrt(3)});

\foreach \y in {0} \draw[color=gray!40!white,fill,ultra thin] plot[samples=360,domain=180:53] ({\y+1/4+cos(\x)/4},{sin(\x)/4})--plot[samples=120,domain=143.1:158.2] ({\y+2/3+cos(\x)/3},{sin(\x)/3})--plot[samples=360,domain=38.2:180] ({\y+1/5+cos(\x)/5},{sin(\x)/5});

\foreach \y in {-1}  \draw[color=gray!40!white,fill,ultra thin] plot[samples=360,domain=0:141.8] ({\y+4/5+cos(\x)/5},{sin(\x)/5})--plot[samples=360,domain=81.8:53.1] ({\y+5/8+cos(\x)/8},{sin(\x)/8})--plot[samples=360,domain=143.1:0] ({\y+5/6+cos(\x)/6},{sin(\x)/6});

\foreach \y in {-1}  \draw[color=gray!40!white,fill,ultra thin] plot[samples=360,domain=81.8:180] ({\y+5/8+cos(\x)/8},{sin(\x)/8})--plot[samples=360,domain=180:127] ({\y+3/4+cos(\x)/4},{sin(\x)/4})--plot[samples=120,domain=36.9:21.8] ({\y+1/3+cos(\x)/3},{sin(\x)/3});

\foreach \y in {-1}  \draw[color=gray!40!white,fill,ultra thin] plot[samples=120,domain=21.8:0] ({\y+1/3+cos(\x)/3},{sin(\x)/3})--plot[samples=120,domain=0:67.4] ({\y+7/12+cos(\x)/12},{sin(\x)/12})--plot[samples=360,domain=157.4:141.8] ({\y+4/5+cos(\x)/5},{sin(\x)/5});

\foreach \y in {0} \draw[color=gray!40!white,fill,ultra thin] plot[samples=120,domain=0:98.2] ({\y+3/8+cos(\x)/8},{sin(\x)/8})--plot[samples=360,domain=38.2:22.6] ({\y+1/5+cos(\x)/5},{sin(\x)/5})--plot[samples=120,domain=112.6:0] ({\y+5/12+cos(\x)/12},{sin(\x)/12});

\foreach \y in {0} \draw[color=gray!40!white,fill,ultra thin] plot[samples=120,domain=98.2:126.9] ({\y+3/8+cos(\x)/8},{sin(\x)/8})--plot[samples=360,domain=36.9:0] ({\y+1/6+cos(\x)/6},{sin(\x)/6})--plot[samples=120,domain=180:158.2] ({\y+2/3+cos(\x)/3},{sin(\x)/3});

\foreach \y in {0} \draw[color=gray!40!white,fill,ultra thin] ({\y},1.5)--({\y},1)--plot[samples=60,domain=90:120] ({\y+cos(\x)},{sin(\x)})--({\y-.5},{sqrt(3)/2})--({\y-.5},1.5);

\draw[->,color=green,line width=1pt] (.5,0)--(.5,.8);
\draw[->,color=green,line width=1pt] (-.5,0)--(-.5,.8);
\draw[->,color=green,line width=1pt] (.5,0)--(.5,.75);
\draw[->,color=green,line width=1pt] (-.5,0)--(-.5,.75);
\draw[->,color=green,line width=1pt] (.5,0)--(.5,.7);
\draw[->,color=green,line width=1pt] (-.5,0)--(-.5,.7);
\draw[color=green,line width=1pt] (.5,0)--(.5,1.5);
\draw[color=green,line width=1pt] (-.5,0)--(-.5,1.5);
\draw[->,color=blue,line width=.5pt] plot[samples=20,domain=100:110] ({cos(\x)/12-5/12},{sin(\x)/12});
\draw[->,color=blue,line width=.5pt] plot[samples=20,domain=120:130] ({cos(\x)/12-5/12},{sin(\x)/12});
\draw[->,color=blue,line width=.5pt] plot[samples=20,domain=100:90] ({cos(\x)/12+5/12},{sin(\x)/12});
\draw[->,color=blue,line width=.5pt] plot[samples=20,domain=80:70] ({cos(\x)/12+5/12},{sin(\x)/12});
\draw[color=blue,line width=1pt] plot[samples=360,domain=0:180] ({cos(\x)/12-5/12},{sin(\x)/12});
\draw[color=blue,line width=1pt] plot[samples=360,domain=0:180] ({cos(\x)/12+5/12},{sin(\x)/12});
\draw[->,color=red,line width=.5pt] plot[samples=20,domain=90:100] ({cos(\x)/6-1/6},{sin(\x)/6});
\draw[->,color=red,line width=.5pt] plot[samples=20,domain=100:90] ({cos(\x)/6+1/6},{sin(\x)/6});
\draw[color=red,line width=1pt] plot[samples=360,domain=0:180] ({cos(\x)/6-1/6},{sin(\x)/6});
\draw[color=red,line width=1pt] plot[samples=360,domain=0:180] ({cos(\x)/6+1/6},{sin(\x)/6});

\foreach \y in {0}
\draw[->]({\y-0.6},0)--({\y+0.6},0)node[right]{$\RE z$};

\draw({1/2},0)node[below]{$^{\frac12} $};
\draw({1/3},0)node[below]{$^{\frac13} $};
\draw(0,-0.25) node[below]{\textbf{(c)}};
\node[color=black] at (0.4,1.375) {$^{ \mathfrak D_{6}}$};
\end{tikzpicture}\hspace{-.4em}\begin{tikzpicture}[scale=2]

\foreach \y in {0}\draw[->]({\y},-.25)--({\y},1.625)node[above]{$\IM z$};
\foreach \y in {0} \draw[color=gray!20!white,fill,ultra thin] ({\y+1/2},1.5)--({\y+1/2},0)--plot[samples=360,domain=0:180] ({\y+5/12+cos(\x)/12},{sin(\x)/12})--plot[samples=360,domain=0:180] ({\y+1/6+cos(\x)/6},{sin(\x)/6})--({\y},{0})--({\y},1.5);

\foreach \y in {0}
\draw[->]({\y-0.15},0)--({\y+0.6},0)node[right]{$\RE z$};

\draw[->,color=green,line width=1pt] (.5,0)--(.5,.8);
\draw[->,color=green,line width=1pt] (.5,0)--(.5,.75);
\draw[color=green,line width=1pt] (.5,0)--(.5,1.5);
\draw[->,color=blue,line width=.1pt] plot[samples=20,domain=60:50] ({5/12+cos(\x)/12},{sin(\x)/12});
\draw[->,color=blue,line width=.1pt] plot[samples=20,domain=140:150] ({5/12+cos(\x)/12},{sin(\x)/12});
\draw[color=blue,line width=1pt] plot[samples=360,domain=0:180] ({cos(\x)/12+5/12},{sin(\x)/12});
\draw[color=brown,line width=1pt] (0,0)--(0,1.5);
\draw[->,color=green,line width=.5pt] plot[samples=20,domain=80:90] ({cos(\x)/6+1/6},{sin(\x)/6});
\draw[->,color=green,line width=.5pt] plot[samples=20,domain=90:100] ({cos(\x)/6+1/6},{sin(\x)/6});
\draw[color=green,line width=1pt] plot[samples=360,domain=0:180] ({cos(\x)/6+1/6},{sin(\x)/6});
\draw[>-,color=brown,line width=1pt] (0,0.8)--(0,0);
\draw[>-,color=brown,line width=1pt] (0,0.75)--(0,0);
\draw[>-,color=brown,line width=1pt] (0,0.7)--(0,0);
\draw[>-,color=brown,line width=1pt] (0,0.2)--(0,0.5);
\draw[>-,color=brown,line width=1pt] (0,0.25)--(0,.5);
\draw[>-,color=brown,line width=1pt] (0,0.15)--(0,.5);

\draw[dotted,thick] plot[samples=180,domain=0:90] ({cos(\x)/sqrt(6)},{sin(\x)/sqrt(6)});

\draw({1/2},0)node[below]{$^{\frac12} $};
\draw({1/3},0)node[below]{$^{\frac13} $};
\draw[fill=black](0,{1/sqrt(6)})circle[radius=0.02]node at (-.1,{1/sqrt(6)-.02}) {$_{\frac{i}{\sqrt6}}$};
\draw[fill=black]({2/5},{1/5/sqrt(6)})circle[radius=0.02];

\node[color=black] at (0.375,1.375) {$^{ \mathfrak D_{6,6}}$};
\draw(0.25,-0.25) node[below]{\textbf{(d)}};

\end{tikzpicture}\hspace{.12em}\begin{tikzpicture}[scale=2]

\foreach \y in {0}\draw[->]({\y},-.25)--({\y},1.625)node[above]{$\IM z$};

\draw[dotted,thick] (.5,1)--(.5,0);
\foreach \y in {0} \draw[color=gray!20!white,fill,ultra thin] ({\y+1/2},1.5)--({\y+1/2},{.5/sqrt(3)})--plot[samples=360,domain=90:150] ({.5+.5/sqrt(3)*cos(\x)},{.5/sqrt(3)*sin(\x)})--plot[samples=360,domain=60:180] ({\y+1/6+cos(\x)/6},{sin(\x)/6})--({\y},{0})--plot[samples=360,domain=0:120] ({\y-1/6+cos(\x)/6},{sin(\x)/6})--plot[samples=360,domain=30:90] ({-.5+.5/sqrt(3)*cos(\x)},{.5/sqrt(3)*sin(\x)})--({\y-1/2},{.5/sqrt(3)})--({\y-1/2},1.5);

\draw[<-](-.5,{.5/sqrt(3)})--(-.5,{.5/sqrt(3)-.08});
\draw[<-](-.5,0)--(-.5,.08);
\node at (-.5,{.25/sqrt(3)}) {\fontsize{4}{4.8}\selectfont${\sqrt3/6}$};
\foreach \y in {0}\draw[->]({\y-0.6},0)--({\y+0.6},0)node[right]{$\RE z$};

\draw[->,color=green,line width=1pt] (.5,{.5/sqrt(3)})--(.5,.8);
\draw[->,color=green,line width=1pt] (-.5,{.5/sqrt(3)})--(-.5,.8);
\draw[->,color=green,line width=1pt] (.5,{.5/sqrt(3)})--(.5,.75);
\draw[->,color=green,line width=1pt] (-.5,{.5/sqrt(3)})--(-.5,.75);
\draw[->,color=green,line width=1pt] (.5,{.5/sqrt(3)})--(.5,.7);
\draw[->,color=green,line width=1pt] (-.5,{.5/sqrt(3)})--(-.5,.7);
\draw[color=green,line width=1pt] (.5,{.5/sqrt(3)})--(.5,1.5);
\draw[color=green,line width=1pt] (-.5,{.5/sqrt(3)})--(-.5,1.5);

\draw[dotted,thick] plot[samples=180,domain=90:180] ({.5+.5/sqrt(3)*cos(\x)},{.5/sqrt(3)*sin(\x)});
\draw[dotted,thick] plot[samples=180,domain=0:90] ({-.5+.5/sqrt(3)*cos(\x)},{.5/sqrt(3)*sin(\x)});
\draw[dotted,thick] plot[samples=360,domain=0:180] ({cos(\x)/6-1/6},{sin(\x)/6});
\draw[dotted,thick] plot[samples=360,domain=0:180] ({cos(\x)/6+1/6},{sin(\x)/6});

\draw[->,color=cyan,line width=.5pt] plot[samples=20,domain=50:60] ({-.5+.5/sqrt(3)*cos(\x)},{.5/sqrt(3)*sin(\x)});
\draw[->,color=cyan,line width=.5pt] plot[samples=20,domain=60:70] ({-.5+.5/sqrt(3)*cos(\x)},{.5/sqrt(3)*sin(\x)});
\draw[->,color=cyan,line width=.5pt] plot[samples=20,domain=130:120] ({.5+.5/sqrt(3)*cos(\x)},{.5/sqrt(3)*sin(\x)});
\draw[->,color=cyan,line width=.5pt] plot[samples=20,domain=120:110] ({.5+.5/sqrt(3)*cos(\x)},{.5/sqrt(3)*sin(\x)});
\draw[color=cyan,line width=1pt] plot[samples=120,domain=90:150] ({.5+.5/sqrt(3)*cos(\x)},{.5/sqrt(3)*sin(\x)});
\draw[color=cyan,line width=1pt] plot[samples=120,domain=30:90] ({-.5+.5/sqrt(3)*cos(\x)},{.5/sqrt(3)*sin(\x)});
\draw[->,color=red,line width=.5pt] plot[samples=20,domain=90:100] ({cos(\x)/6-1/6},{sin(\x)/6});
\draw[->,color=red,line width=.5pt] plot[samples=20,domain=100:90] ({cos(\x)/6+1/6},{sin(\x)/6});
\draw[color=red,line width=1pt] plot[samples=240,domain=0:120] ({cos(\x)/6-1/6},{sin(\x)/6});
\draw[color=red,line width=1pt] plot[samples=240,domain=60:180] ({cos(\x)/6+1/6},{sin(\x)/6});

\draw({1/2},0)node[below]{$^{\frac12} $};
\draw({1/3},0)node[below]{$^{\frac13} $};

\node[color=black] at (0.375,1.375) {$^{ \mathfrak D_{6,3}}$};
\draw(0,-0.25) node[below]{\textbf{(e)}};\end{tikzpicture}\hspace{.12em}\begin{tikzpicture}[scale=2]

\foreach \y in {0}\draw[->]({\y},-.25)--({\y},1.625)node[above]{$\IM z$};
\foreach \y in {0} \draw[color=gray!20!white,fill,ultra thin] ({\y+1/2},1.5)--({\y+1/2},0)-- plot[samples=200,domain=0:101.5] ({5/12+cos(\x)/12},{sin(\x)/12})--plot[samples=360,domain=11.5:168.5] ({cos(\x)/sqrt(6)},{sin(\x)/sqrt(6)})-- plot[samples=200,domain=78.5:180] ({-5/12+cos(\x)/12},{sin(\x)/12})--({\y-1/2},0)--({\y-1/2},1.5);

\foreach \y in {0}\draw[->]({\y-0.6},0)--({\y+0.6},0)node[right]{$\RE z$};

\draw[->,color=green,line width=1pt] (.5,0)--(.5,.8);
\draw[->,color=green,line width=1pt] (-.5,0)--(-.5,.8);
\draw[->,color=green,line width=1pt] (.5,0)--(.5,.75);
\draw[->,color=green,line width=1pt] (-.5,0)--(-.5,.75);
\draw[->,color=green,line width=1pt] (.5,0)--(.5,.7);
\draw[->,color=green,line width=1pt] (-.5,0)--(-.5,.7);
\draw[color=green,line width=1pt] (.5,0)--(.5,1.5);
\draw[color=green,line width=1pt] (-.5,0)--(-.5,1.5);

\draw[dotted,thick] plot[samples=180,domain=90:180] ({cos(\x)/sqrt(6)},{sin(\x)/sqrt(6)});
\draw[dotted,thick] plot[samples=180,domain=0:90] ({cos(\x)/sqrt(6)},{sin(\x)/sqrt(6)});
\draw[dotted,thick] plot[samples=360,domain=0:180] ({5/12+cos(\x)/12},{sin(\x)/12});
\draw[dotted,thick] plot[samples=360,domain=0:180] ({-5/12+cos(\x)/12},{sin(\x)/12});
\draw[->,color=blue,line width=.5pt] plot[samples=20,domain=130:140] ({-5/12+cos(\x)/12},{sin(\x)/12});
\draw[->,color=blue,line width=.5pt] plot[samples=20,domain=110:120] ({-5/12+cos(\x)/12},{sin(\x)/12});
\draw[->,color=blue,line width=.1pt] plot[samples=20,domain=70:60] ({5/12+cos(\x)/12},{sin(\x)/12});
\draw[->,color=blue,line width=.1pt] plot[samples=20,domain=50:40] ({5/12+cos(\x)/12},{sin(\x)/12});
\draw[color=blue,line width=1pt] plot[samples=200,domain=0:101.5] ({5/12+cos(\x)/12},{sin(\x)/12});
\draw[color=blue,line width=1pt] plot[samples=200,domain=78.5:180] ({-5/12+cos(\x)/12},{sin(\x)/12});
\draw[->,color=magenta,line width=.5pt] plot[samples=20,domain=130:120] ({cos(\x)/sqrt(6)},{sin(\x)/sqrt(6)});
\draw[->,color=magenta,line width=.5pt] plot[samples=20,domain=150:160] ({cos(\x)/sqrt(6)},{sin(\x)/sqrt(6)});
\draw[->,color=yellow,line width=.5pt] plot[samples=20,domain=30:20] ({cos(\x)/sqrt(6)},{sin(\x)/sqrt(6)});
\draw[->,color=yellow,line width=.5pt] plot[samples=20,domain=50:60] ({cos(\x)/sqrt(6)},{sin(\x)/sqrt(6)});
\draw[color=yellow,line width=1pt] plot[samples=360,domain=11.5:90] ({cos(\x)/sqrt(6)},{sin(\x)/sqrt(6)});
\draw[color=magenta,line width=1pt] plot[samples=360,domain=90:168.5] ({cos(\x)/sqrt(6)},{sin(\x)/sqrt(6)});

\draw[dotted,thick] ({1/3},0)--({1/3},{sqrt(2)/6});
\draw[dotted,thick] ({-1/3},0)--({-1/3},{sqrt(2)/6});
\draw[fill=yellow,draw=yellow]({1/3},{sqrt(2)/6})circle[radius=0.02];
\draw[fill=magenta,draw=magenta]({-1/3},{sqrt(2)/6})circle[radius=0.02];

\draw({1/2},0)node[below]{$^{\frac12} $};
\draw({1/3},0)node[below]{$^{\frac13} $};

\draw[fill=black](0,{1/sqrt(6)})circle[radius=0.02]node at (0,{1/sqrt(6)+.125}) {$_{\frac{i}{\sqrt6}}$};
\node[color=black] at (0.375,1.375) {$^{ \mathfrak D_{6,2}}$};
\draw(0,-0.25) node[below]{\textbf{(f)}};
\end{tikzpicture}

\caption{(Adapted from  \cite[Fig.~61]{FrickeKlein}.) \textbf{a} Fundamental domain $ \mathfrak D$ of  $ \Gamma_0(1)=SL(2,\mathbb Z)$. The moduli space $ Y_0(1)(\mathbb C)=SL(2,\mathbb Z)\backslash \mathfrak H$ is  a quotient space of $ \mathfrak D$ that identifies the corresponding sides of the boundary   $\partial \mathfrak D$  along  the  \textit{arrows}. \textbf{b} Tessellation of the upper half-plane $ \mathfrak H$ by   successive translations  [generator $\hat T=\hat \tau^{-1}:z\mapsto z+1 $]
   and inversions    [generator $\hat S=\hat S^{-1}:z\mapsto-1/z$] of the fundamental domain $\mathfrak D$. Each tile  is  subdivided and painted in \textit{gray} or \textit{white}   according as the pre-image satisfies  $ \RE z<0$ or $ \RE z>0$ in the fundamental domain $ \mathfrak D$. \textbf{c} Fundamental domain $ \mathfrak D_6$ of $\Gamma_0(6)$, dissected with $ SL(2,\mathbb Z)$-tiles  (cf.~panel \textbf{b}). Gluing the three pairs of boundary sides of $ \mathfrak D_6$  along the \textit{arrows}, one obtains the moduli space $ Y_0(6)(\mathbb C)=\Gamma_0(6)\backslash\mathfrak H$.  \textbf{d}--\textbf{f} Fundamental domains $ \mathfrak D_{6,k}$ for the Chan--Zudilin groups $ \Gamma_0(6)_{+k}=\langle\Gamma_0(6),\hat W_k\rangle$, where $ \hat W_2z= (2z-1)/(6z-2)$, $ \hat W_3z= (3z-2)/(6z-3)$, and $ \hat W_6z= -1/(6z)$.
 \label{fig:fun_domains}
}
       \end{figure}

Following the notation of Chan--Zudilin \cite{ChanZudilin2010}, we write $ \hat W_3=\frac{1}{\sqrt{3}}\left(\begin{smallmatrix}3&-2\\6&-3\end{smallmatrix}\right)$ and construct a group $ \Gamma_0(6)_{+3}=\langle\Gamma_0(6),\hat W_3\rangle$  by adjoining $ \hat W_3$ to $ \Gamma_0(6)$. This group is of particular importance to the following motivic integral  \cite[\S2]{BlochKerrVanhove2015}: \begin{align}\begin{split}&
 \int_0^\infty I_0(\sqrt{u}t)[K_0(t)]^4t\D t\\={}&\frac{1}{8}\int_0^\infty\frac{\D X}{X}\int_0^\infty\frac{\D Y}{Y}\int_0^\infty\frac{\D Z}{Z}\frac{1}{(1+X+Y+Z)\left( 1+\frac{1}{X} +\frac{1}{Y}+\frac{1}{Z}\right)-u}.\end{split}
\end{align} As pointed out in Verrill's thesis \cite[Theorems 1 and 2]{Verrill1996}, the differential equation $\widetilde L_3f(u)=0 $ (cf.~Table \ref{tab:VanhoveL}) is the Picard--Fuchs equation for a pencil of $K3$ surfaces:\begin{align}
\mathscr X_{A_3}:(1+X+Y+Z)\left( 1+\frac{1}{X} +\frac{1}{Y}+\frac{1}{Z}\right)=u,
\end{align}whose monodromy group is isomorphic to $\overline{ \Gamma_0(6)_{+3}}$,  the image of $\Gamma_0(6)_{+3} $ after quotienting out by scalars. As a consequence, the general solutions to  $\widetilde L_3f(u)=0 $ admit a modular parametrization \begin{align}
f(u)=Z_{6,3}(z)(c_0+c_1z+c_2z^2),
\end{align} where $ c_0,c_1,c_2\in\mathbb C$ are constants, and \begin{align}
u=-64X_{6,3}(z):={}&-\left[\frac{ 2\eta (2z ) \eta (6 z )}{\eta (z ) \eta (3z)}\right]^{6},\\Z_{6,3}(z):={}&\frac{[\eta(z)\eta(3z)]^{4}}{[\eta(2z)\eta(6z)]^{2}}.
\end{align}Here,  $ X_{6,3}(z)$ is a modular function  on  $\Gamma_0(6)_{+3}$ \cite[(2.2)]{ChanZudilin2010}, while $ Z_{6,3}(z)$ is a modular form of weight 2 and level 6 \cite[(2.5)]{ChanZudilin2010}.

Since $ \int_0^\infty J_0(\sqrt{-u}t)I_0(t)[K_0(t)]^3t\D t$ is annihilated by Vanhove's operator $ \widetilde L_3$, we can establish the following modular parametrization  \begin{align}
\int_0^\infty J_0\left(8\sqrt{\smash[b]{X_{6,3}(z)}}t\right)I_0(t)[K_0(t)]^3t\D t=\frac{\pi^{2}}{16}Z_{6,3}(z)\label{eq:IKKK_Hankel}
\end{align}through asymptotic analysis of both sides  near the infinite cusp $  z\to i\infty$ [whereupon the left-hand side tends to $ \int_0^\infty I_0(t)[K_0(t)]^3t\D t=\IKM(1,3;1)$ and the right-hand side tends to $ \frac{\pi^2}{16}=\IKM(1,3;1)$]. Here, the positive $\IM z $-axis corresponds to $ \sqrt{-u}=8\sqrt{\smash[b]{X_{6,3}(z)}}\in(0,\infty)$. In a similar fashion, one can show that\begin{align}
\int_0^\infty J_{0}\left(8\sqrt{\smash[b]{X_{6,3}(z)}}t\right)[I_{0}(t)]^{2}[K_0(t)]^{2}t\D t={}&\frac{\pi z}{4i}Z_{6,3}(z)\label{eq:IIKK_Hankel}
\end{align} holds for $ z/i>0$. Now, we can prove Theorem \ref{thm:specL}\ref{itm:8Bessel} by throwing  \eqref{eq:IKKK_Hankel}--\eqref{eq:IIKK_Hankel} into the Parseval--Plancherel identity for Hankel transforms \cite[(16)]{BBBG2008} \begin{align}
\int_0^\infty f(t)g(t)t\D t=\int_0^\infty\left[ \int_0^\infty J_0(xt)f(t)t\D t \right] \left[ \int_0^\infty J_0(x\tau)g(\tau)\tau\D \tau \right] x\D x,
\end{align} and noting that \cite[Theorem 5.1.1]{Zhou2017WEF}\begin{align}
[Z_{6,3}(z)]^{2}\frac{\D  X_{6,3}(z)}{\D z}=2\pi i f_{6,6}(z).
\end{align}

Actually, we can  say a little more about the 8-Bessel problem than what has been stated in Theorem \ref{thm:specL}\ref{itm:8Bessel}. With heavy use of Wick rotations and integrations of $ f_{6,6}(z)z^n,n\in\{0,1,2,3,4\}$ over the boundary $ \partial\mathfrak  D_{6,3}$ of the fundamental domain $\mathfrak D_{6,3}$ (Fig.~\ref{fig:fun_domains}\textit{e}), one may show that \cite[\S5]{Zhou2017WEF}\begin{align}
\frac{L(f_{6,6},5)}{L(f_{6,6},3)} =\frac{2\pi^{2}}{21}.
\end{align}Comparing this to Theorem \ref{thm:specL}\ref{itm:8Bessel}, one  confirms a sum rule $9\pi^2\IKM(4,4;1)=14\IKM(2,6;1)$, which was originally proposed in 2008  \cite[at the end of \S6.3, between (228) and (229)]{BBBG2008}.
  \section{Some linear sum rules of Feynman diagrams\label{sec:Crandall}}
The contour integral in \eqref{eq:HHM}  is no longer convergent when $n\in\mathbb Z\cap[m,\infty)$, so the methods in \S\ref{subsec:Wick} do not directly apply to Theorem \ref{thm:B3G_Crandall}\ref{itm:Crandall}, which involves Bessel moments $\IKM(a,b;n)$ with high orders $n\geq(a+b-2)/2$. In \cite[\S3]{HB1}, I used  a real-analytic approach (based on Hilbert transforms), to circumvent divergent contour integrals while handling  Theorem \ref{thm:B3G_Crandall}\ref{itm:Crandall}. After email exchanges
with Mark van Hoe{\ij} on Oct.~24, 2017, about the asymptotic expansion of $ [I_0(x)K_0(x)]^4$ for large and positive $ x$ (see van Hoe\ij's update on \cite{OEISA262961}, dated  Oct.~23, 2017), I\ realized that the divergence problem in the complex-analytic approach can be amended by subtracting Laurent polynomials from $ [H_0^{(1)}(z)H_0^{(2)}(z)]^m$. This amendment is described in the lemma below. \begin{lemma}[Asymptotic expansions and Bessel moments]\label{lm:vanHoeij}We have the following asymptotic expansion as $|z|\to\infty,-\pi<\arg z<\pi $:\begin{align}\begin{split}&
\left(\frac{\pi}{2}\right)^{2m}[H_0^{(1)}(z)H_0^{(2)}(z)]^m\\={}&\sum_{n=1}^N\frac{(-1)^{n+1}}{z^{2n+m-2}}\int_0^\infty \frac{[\pi I_0(t)+i K_{0}(t)]^{m}-[\pi I_0(t)-i K_{0}(t)]^{m}}{\pi i}[K_0(t)]^{m}t^{2n+m-3}\D t\\{}&+O\left( \frac{1}{|z|^{2N+m}} \right),\end{split}
\end{align}where $ m,N\in\mathbb Z_{>0}$.\end{lemma}\begin{proof}From \eqref{eq:H0_asympt}, we know that  as $|z|\to\infty,-\pi<\arg z<\pi $, there exist  certain constant coefficients $ a_{m,n}$ such that the following relation holds:\begin{align}
F_{m,N}(z):=[H_0^{(1)}(z)H_0^{(2)}(z)]^m-\sum_{n=1}^N\frac{a_{m,n}}{z^{2n+m-2}}=O\left( \frac{1}{|z|^{2N+m}} \right).
\end{align} To determine   $ a_{m,N}$, we consider \begin{align}\lim_{\varepsilon\to0^+}
\lim_{T\to\infty}\left( \int_{-iT}^{-i\varepsilon} +\int_{C_{\varepsilon}}+\int_{i\varepsilon}^{iT}\right)F_{m,N}(z)z^{2N+m-3}\D z,
\end{align}where $C_{\varepsilon} $ is a semi-circular arc in the right half-plane, joining $ -i\varepsilon$ to $ i\varepsilon$. For each fixed $ \varepsilon>0$, the contour integral in question tends to zero, as $ T\to\infty$, because we can close the contour to the right. Recalling \eqref{eq:H1H2Imz}, and integrating  the Laurent polynomial over  $C_{\varepsilon} $, we arrive at the claimed result.  \qed\end{proof}

Before moving onto the proof of Theorem \ref{thm:B3G_Crandall}\ref{itm:Crandall} in the next proposition, we point out that one can also generalize the method in the last lemma into other cancelation formulae. For example, in \cite[Lemma 3.3]{Zhou2018LaportaSunrise}, we used  a vanishing contour integral\begin{align}
\lim_{T\to\infty}\int_{-iT}^{iT}H_0^{(1)}(z)H_0^{(2)}(z)\left\{[H_0^{(1)}(z)H_0^{(2)}(z)]^{2}-\frac{4}{\pi^{2}z^2}\right\}z^3\D z=0
\end{align}to prove \begin{align}
\int_0^\infty I_0(t)[K_0(t)]^5t^{3}\D t={}&\frac{\pi^2}{3}\int_0^\infty I_0(t)K_0(t)\left\{[I_{0}(t)]^{2} [K_0(t)]^2-\frac{1}{4t^{2}}\right\}t^{3}\D t,
\end{align}which paved way for the verification of a conjecture  \cite[(1.11)]{Zhou2018LaportaSunrise} due to  Laporta \cite[(29)]{Laporta:2017okg} and Broadhurst (private communication on Nov.~10, 2017).

Thanks to van Hoeij's observation that led to Lemma \ref{lm:vanHoeij}, we see that the expression $C_{m,n}$ in \eqref{eq:Crandall_int} evaluates to a rational number [cf.~\eqref{eq:H0_asympt}], and these sequences of rational numbers satisfy a   discrete convolution relation with respect to the power $m$. To show that  $ C_{m,n}$ is in fact  a positive integer, it now suffices to prove that, for each  $\ell\in\mathbb Z_{>0}$, \begin{align}C_{1,\ell}={}&\frac{ [(2 \ell-2)!]^3}{2^{2\ell-2}[(\ell-1)!]^4}=[(2\ell-3)!!]^{2}{2 \ell-2\choose \ell-1}\intertext{and}
C_{2,\ell}={}&\frac{1}{2^{4(\ell-1)}}\sum_{k=1}^{\ell}  \frac{ [(2 \ell-2k)!]^3}{ [( \ell-k)!]^4}\frac{ [(2k-2)!]^3}{ [(k-1)!]^4}
\end{align}are both integers. Here, we have  $ {n\choose k}:=\frac{n!}{k!(n-k)!}\in\mathbb Z$ for $ n\in\mathbb Z_{\geq0},k\in\mathbb Z\cap[0,n]$,  and   $ (2n-1)!!:=(2n)!/(n!2^n)\in\mathbb Z$ for $n\in\mathbb Z_{\geq0}$, so the statement  $C_{1,\ell} \in\mathbb Z$ holds true. The integrality of $C_{2,\ell} $ will be explained below.\begin{proposition}
[An integer sequence]For each $ \ell\in\mathbb Z_{>0}$, the number $ \alpha_\ell:=C_{2,\ell}$ is a positive integer.\end{proposition}\begin{proof}In \cite[Theorem 3.1]{Rogers2009}, Mathew Rogers has effectively shown that the following identity holds for $|u|$ sufficiently small:
\begin{align}\sum_{\ell=1}^\infty\frac{\alpha_{\ell+1}-\ell^{2}\alpha_{\ell}}{(\ell !)^2}u^\ell=3\sum_{n=1}^\infty[(2n-1)!!]^2{3n-1\choose2n}\frac{1}{(n!2^n)^2}\frac{u^{2n}}{(1-u)^{3n}}.
\end{align} Comparing the coefficients of $ u^n$ on both sides, we see that, for each $ n\in\mathbb Z_{>0}$, the expression $ \alpha_{n+1}-n^{2}\alpha_{n}$  equals a sum of finitely many  terms, each of which is  an integer multiple of $ (k!!)^{2}\in\mathbb Z$ for a certain odd positive integer $k$ less than $n$.  Therefore, we have $ \alpha_1=1,\alpha_{\ell+1}-\ell^{2}\alpha_{\ell}\in\mathbb Z$ for $ \ell\in\mathbb Z_{>0}$, which entails the claimed result.
\qed\end{proof}
\section{Some non-linear sum rules of Feynman diagrams\label{sec:BMdet}}
As we did in \S\ref{sec:Crandall}, we will build non-linear sum rules of Feynman diagrams without evaluating individual Bessel moments in closed form. In what follows, we describe a key step towards the proof of
Broadhurst--Mellit
determinant formulae (Theorem \ref{thm:BMdet}),  namely,  the asymptotic analysis of the Wro\'nskians $ \Omega_{2k-1}(u)$ and $ \omega_{2k}(u)$ introduced in \S\ref{subsec:VanhoveL}.

As in   \cite[\S4]{Zhou2017BMdet}, we differentiate \eqref{eq:mu_defn} with respect to $u$\ and define\begin{align}\acute
\mu^\ell_{k,j}(u):=2\sqrt{u}D^1\mu^\ell_{k,j}(u),\quad \forall j\in\mathbb Z\cap[1,2k-1].
\end{align}Through iterated applications of the Bessel differential equations $ (uD^2+D^1)I_{0}(\sqrt{u}t)=\frac{t^2}{4}I_0(\sqrt{u}t)$ and $(uD^2+D^1)K_{0}(\sqrt{u}t)=\frac{t^2}{4}K_0(\sqrt{u}t) $, we can verify \begin{align}
(2\sqrt{u})^{(k-1)(2k-1)}\Omega_{2k-1}(u)={}&\det \begin{pmatrix}\mu_{k,1}^1(u) & \cdots & \mu_{k,2k-1}^1(u) \\
\acute\mu_{k,1}^1(u) & \cdots & \acute \mu_{k,2k-1}^1(u) \\
\multicolumn{3}{c}{\cdots\cdots\cdots\cdots\cdots\cdots\cdots} \\
\mu_{k,1}^k(u) & \cdots & \mu_{k,2k-1}^k(u)  \\
\end{pmatrix},
\end{align}where the $\mu$ (resp.~$ \acute \mu$) terms occupy the odd-numbered (resp.~even-numbered)
rows. Since $ W[I_0(u),K_0(u)]=-I_0(u)K_{1}(u)-K_{0}(u)I_{1}(u)=-1/u$, we can show that  \begin{align}
\left\{ \begin{array}{l}
\mu^\ell_{k,j}(1)=\mu^\ell_{k,k+j-1}(1), \\\acute
\mu^\ell_{k,k+j-1}(1)-\acute \mu^\ell_{k,j}(1)=-\mu_{k-1,j-1}^\ell(1)
\end{array} \right.
\end{align}for all $  j\in\mathbb Z\cap[2,k]$.
Thus, we obtain, after column eliminations and row bubble sorts,{\allowdisplaybreaks\begin{align}\begin{split}&
2^{(k-1)(2k-1)}\Omega_{2k-1}(1)\\={}&\det \begin{pmatrix}\mu_{k,1}^1 (1)& \cdots&\mu^1_{k,k}(1)&0&\cdots & 0 \\
\acute\mu_{k,1}^1(1) & \cdots&\acute\mu^1_{k,k}(1)&-\mu_{k-1,1}^1(1)&\cdots  &  -\mu_{k-1,k-1}^1(1) \\
\multicolumn{6}{c}{\cdots\cdots\cdots\cdots\cdots\cdots\cdots\cdots\cdots\cdots\cdots\cdots\cdots\cdots\cdots\cdots} \\
\mu_{k,1}^k (1)& \cdots & \mu^k_{k,k}(1)&0&\cdots &0  \\
\end{pmatrix}\\={}&(-1)^{\frac{k(k-1)}{2}}\det\left( \begin{array}{c|c}\begin{array}{ccc}
 &  &  \\
 & \raisebox{-0.25\height}{\resizebox{1.75\width}{1.75\height}{$\mathbf M_k^{\mathrm T}$}} &  \\
 &  &  \\
\end{array}&\raisebox{-0.25\height}{\resizebox{1.75\width}{1.75\height}{$\mathbf O_{}$}}\\\hline\begin{array}{ccc}\vspace{-0.75em}\\\acute
\mu_{k,1}^1(1) & \cdots&\acute\mu^1_{k,k}(1) \\
\multicolumn{3}{c}{\cdots\cdots\cdots\cdots\cdots\cdots\cdots}  \\
\acute
\mu_{k,1}^{k-1}(1) & \cdots&\acute\mu^{k-1}_{k,k}(1) \\
\end{array}&\raisebox{-0.25\height}{\resizebox{1.75\width}{1.75\height}{$-\mathbf M_{k-1}^{\mathrm T}$}}\end{array} \right),\end{split}
\end{align}}which factorizes into\begin{align}
\Omega_{2k-1}(1)=(-1)^{\frac{(k-1)(k-2)}{2}}\frac{\det\mathbf M_{k-1}}{2^{(k-1)(2k-1)}}\det\mathbf M_k\label{eq:Omega1_fac}
\end{align} for each $ k\in\mathbb Z_{\geq2}$.
By a similar procedure (see \cite[Proposition 4.4]{Zhou2017BMdet} for detailed asymptotic analysis), one can show that \begin{align}
\lim_{u\to0^+}u^{k(2k-1)/2}\Omega_{2k-1}(u)=(-1)^{\frac{(k-1)(k-2)}{2}}\frac{k[\varGamma(k/2)]^{2}}{(2k+1)}\frac{(\det\mathbf N_{k-1})^2}{2^{(k-1)(2k-1)+1}}.
\end{align}Consequently, the evolution equation in \eqref{eq:Omega_evolv} admits a solution\begin{align}
\Omega_{2k-1}(u)=\frac{(-1)^{\frac{(k-1)(k-2)}{2}}k[\varGamma(k/2)]^{2}}{u^{k(2k-1)/2}(2k+1)}\frac{(\det\mathbf N_{k-1})^2}{2^{(k-1)(2k-1)+1}}\prod _{j=1}^k\left[\frac{(2j)^2}{(2j)^2-u}\right]^{k-\frac{1}{2}}\label{eq:Omega_2k_1_u_rational}
\end{align}for $ u\in(0,1]$.

Comparing \eqref{eq:Omega1_fac} and \eqref{eq:Omega_2k_1_u_rational}, we arrive at
\begin{align} \det\mathbf M_{k-1}\det\mathbf M_k={}&\frac{k [\varGamma (k/2)]^2(\det\mathbf N_{k-1})^2}{2 (2 k+1)}\prod _{j=1}^k \left[\frac{(2 j)^2}{(2 j)^2-1}\right]^{k-\frac{1}{2}},\label{eq:detM_rec}\intertext{for all $ k\in\mathbb Z_{\geq2}$.
A similar service \cite[\S4.3]{Zhou2017BMdet} on $ \omega_{2k}(u)$  then brings us}\det\mathbf N_{k-1}\det\mathbf N_k={}&\frac{2k+1}{k+1}\frac{(\det\mathbf M_{k})^2}{(k-1)!}\prod _{j=2}^{k+1} \left[\frac{(2 j-1)^2}{(2 j-1)^2-1}\right]^{k}.\label{eq:detN_rec}\end{align}The last pair of equations, together with the initial conditions  $ \det\mathbf M_1=\IKM(1,2;1)=\frac{\pi}{3\sqrt{3}}$  \cite[(23)]{BBBG2008} and $ \det\mathbf N_1=\IKM(1,3;1)=\frac{\pi^2}{2^{4}}$ \cite[(55)]{BBBG2008}, allow us to prove Theorem \ref{thm:BMdet} by induction.

As a by-product, we see from  \eqref{eq:Omega_2k_1_u_rational} that $ \Omega_{2k-1}(u)=W[\mu^1_{k,1}(u),\dots,\mu^1_{k,2k-1}(u)]$ is non-vanishing for $ u\in(0,1]$. Therefore,  the functions $ \mu^1_{k,1}(u),\dots,\mu^1_{k,2k-1}(u)$ (restricted to  the interval $ (0,1]$) form a basis for the kernel space of  $ \widetilde L_{2k-1}$. Consequently, for each $ k\in\mathbb Z_{\geq2}$, the function
$ p_{2k}(\sqrt{u})/\sqrt{u},0<u\leq 1$ (where $ p_{2k}(x)=\int_0^\infty J_0(xt)[J_0(t)]^{2k} xt\D t$ is Kluyver's probability density) is  an $ \mathbb R$-linear combination of the functions $ \mu^1_{k,1}(u),\dots,\mu^1_{k,2k-1}(u)$. Unlike  our statement  in Theorem \ref{thm:p_odd}, where the Bessel moment representation of  $ p_{2j+1}(x),j\in\mathbb Z_{\geq2}$ leaves a convergent Taylor expansion for $ 0\leq x\leq 1$, the representation of $ p_{2k}(x),0\leq x\leq 1$ through a linear combination of  Bessel moments may involve $O(x\log x) $ singularities in the $x\to0^+ $ regime, attributable to the Bessel function $K_0$. Such logarithmic singularities had been  previously studied by Borwein--Straub--Wan--Zudilin \cite{BSWZ2012}.
\section{Critical values of modular $L$-functions and multi-loop Feynman diagrams\label{sec:L-value}}As in the proof of Theorem \ref{thm:specL}\ref{itm:8Bessel} in \S\ref{subsec:mod_form}, we need to fuse Hankel transforms in the Parseval--Plancherel identity to prove \eqref{eq:IKM241}.

Fusing the following Hankel transform (cf.~\cite[(4.1.16)]{Zhou2017WEF})\begin{align}
\int_0^\infty J_0\left(\frac{3[ \eta (w)]^2 [\eta (6 w)]^4}{[\eta (3 w)]^2[\eta (2 w)]^4 }it\right)I_0(t)[K_0(t)]^2t\D t=\frac{\pi}{3\sqrt{3}}\frac{\eta (3 w)[ \eta (2 w)]^6}{[\eta (w)]^3 [\eta (6 w)]^2}
\end{align}\big(where $ w=\frac{1}{2}+iy,y>0$ corresponds to $ 0<\frac{3[ \eta (w)]^2 [\eta (6 w)]^4}{[\eta (3 w)]^2[\eta (2 w)]^4 }i<\infty$\big) with itself, we obtain (cf.~\cite[Proposition 4.2.1]{Zhou2017WEF})\begin{align}
\IKM(2,4;1)=\frac{\pi^3i}{3}\int_{\frac12}^{\frac12+i\infty}f_{4,6}(w)\D w.
\end{align} This is not quite the statement in \eqref{eq:IKM241} yet, as the integration path still sits on the ``wrong'' portion of $\partial \mathfrak D_{6,6} $ (Fig.~\ref{fig:fun_domains}\textit{d}). To compensate for this, we need another Hankel fusion, together with some modular transforms on the Chan--Zudilin group $ \Gamma_0(6)_{+2}$, to construct an identity \cite[Proposition 4.2.2]{Zhou2017WEF}: \begin{align}
\JYM(6,0;1)={}&\frac{12}{\pi i}\int_{0}^{i\infty}f_{4,6}(w)\D  w-\frac{6}{\pi i}\int_{\frac12}^{\frac{1}{2}+i\infty}f_{4,6}(w)\D w.\end{align}Now that Wick rotation brings us $ \IKM(2,4;1)=\frac{\pi^{4}}{30}\JYM(6,0;1)$ \cite[(4.1.1)]{Zhou2017WEF}, we can deduce  \eqref{eq:IKM241} from the last two displayed equations.

It takes slightly more effort to verify \eqref{eq:IKM151IKM331}. Towards this end, we need a ``Hilbert cancelation formula'' \cite[Lemma 4.2.4]{Zhou2017WEF}\begin{align}
\int_0^\infty \left[ \int_0^\infty J_0(xt) F(t)t\D t \right]\left[\int_0^\infty Y_0(x\tau) F(\tau)\tau\D \tau\right]x\D x
=0
\end{align}for functions $F(t),t>0$ satisfying certain growth bounds, along with  modular parametrizations of some generalized Hankel transforms, such as (cf.~\cite[(4.1.31)]{Zhou2017WEF})\begin{align}\begin{split}&
\int_0^\infty J_0\left(\frac{3[ \eta (w)]^2 [\eta (6 w)]^4}{[\eta (3 w)]^2[\eta (2 w)]^4 }it\right)[K_0(t)]^3t\D t\\&-\frac{3\pi}{2}\int_0^\infty Y_0\left(\frac{3[ \eta (w)]^2 [\eta (6 w)]^4}{[\eta (3 w)]^2[\eta (2 w)]^4 }it\right)I_0(t)[K_0(t)]^2t\D t\\={}&\frac{\pi^{2}(2w-1)}{2\sqrt{3}i}\frac{\eta (3 w)[ \eta (2 w)]^6}{[\eta (w)]^3 [\eta (6 w)]^2}\end{split}
\end{align} for  $ w=\frac{1}{2}+iy,y>0$. We refer our readers to \cite[Proposition 4.1.3 and Theorem 4.2.5]{Zhou2017WEF} for detailed computations that lead to  \eqref{eq:IKM151IKM331}. \section{Outlook\label{sec:open}}\subsection{Broadhurst's $p$-adic heuristics}
In \S\ref{subsec:arith_BM}, the modular forms $ f_{3,15}$, $ f_{4,6}$ and $ f_{6,6}$ were not picked randomly, but were discovered by Broadhurst via some deep insights into $p$-adic analysis and \'etale cohomology \cite{SGA4.5,SGA5}. In short, Broadhurst's computations of Bessel moments over finite fields  led him to local factors in the Hasse--Weil zeta functions, which piece together into the modular $L$-functions, namely, $ L(f_{3,15},s)$ for the 5-Bessel problem, $ L(f_{4,6},s)$ for the 6-Bessel problem, and $ L(f_{6,6},s)$ for the 8-Bessel problem.

On the arithmetic side,
 Broadhurst investigated Kloosterman moments (``Bessel moments over finite fields''), with extensive numerical experiments  \cite[\S\S2--6]{Broadhurst2016}.
A Bessel function over a finite field \cite{Robba1986pBessel}, with respect to the variable $a\in\mathbb F_{q^{\vphantom k}}=\mathbb F_{p^k}$, is defined by the following Kloosterman sum: \begin{align}
\Kl_2(\mathbb F_{p^k},a):={}&\sum_{x_1,x_2\in\mathbb F^\times_q,x_1x_2= a}e^{\frac{2\pi  i}{p}\Tr_{\mathbb F_{q}/\mathbb F_p}(x_1+x_2)}=\sum_{x\in\mathbb F^\times_q}e^{\frac{2\pi  i}{p}\Tr_{\mathbb F_{q}/\mathbb F_p}\left( x+\frac{a}{x} \right)}
\end{align}where the Frobenius trace $ \Tr_{\mathbb F_{q}/\mathbb F_p}$ acts on  an element $ z\in\mathbb F_q$ as $\Tr_{\mathbb F_{q}/\mathbb F_p}(z):=\sum_{j=0}^{k-1}z^{p^j} $. Writing $ \Kl_2(\mathbb F_{p^k},a)=-\alpha_a-\beta_a$ where $ \alpha_a\beta_a=q$, and introducing the $n$-th symmetric power  $ \Kl_2^n:=\Sym^{n}(\Kl_2)$ as $ \Kl_2^n(\mathbb F_{p^k},a):=\sum_{j=0}^n\alpha_a^j\beta_a^{n-j}$, we may further define Bessel moments over a finite field as the following Kloosterman moments:\begin{align}
S_{n}(q):=\sum_{a\in\mathbb F^\times_q}\Kl_2^n(\mathbb F_{p^k},a)=\sum_{a\in\mathbb F^\times_q}\sum_{j=0}^n\alpha_a^j\beta_a^{n-j}.
\end{align}
With $c_n(q)=-\frac{1+S_n(q)}{q^2} $ for a prime power $q=p^k$, one defines the Hasse--Weil local factor by a formula\begin{align}
 Z_n(p,T):=\exp\left( -\sum_{k=0}^\infty\frac{c_n(p^k)}{k}T^k \right).
\end{align}Following the notations of  Fu--Wan \cite{FuWan2008MathAnn}, we set $ L_p(\mathbb P^1_{\mathbb F_p}\smallsetminus\{0,\infty\},\Sym^n(\Kl_2),s)=1/Z_n(p,p^{-s})$, and define the Hasse--Weil zeta function\begin{align}
\zeta _{n,1}(s):=\prod_pL_p(\mathbb P^1_{\mathbb F_p}\smallsetminus\{0,\infty\},\Sym^n(\Kl_2),s)=\prod_p\frac{1}{Z_{n}(p,p^{-s})},
\end{align}where the product runs over all the primes. It is known that $ \zeta_{5,1}(s)=L(f_{3,15},s)$  \cite{PetersTopVlugt1992}  and $ \zeta_{6,1}(s)=L(f_{4,6},s)$ \cite{Hulek2001}. The structure of $ \zeta_{7,1}(s)$, which involves a Hecke eigenform of weight $3$ and level $525$,  had been conjectured by Evans \cite[Conjecture 1.1]{Evans2010b}, before being completely verified by Yun \cite[\S4.7.7]{Yun2015}.
The story for the 8-Bessel problem is much  more convoluted (see \cite[Theorem 4.6.1 and Appendix B]{Yun2015} as well as \cite[\S7.6]{Broadhurst2016}).

On the geometric side, Broadhurst's $ L$-functions $ L(f_{3,15},s)$ and $ L(f_{4,6},s)$ are closely related to the \'etale cohomologies of certain Calabi--Yau manifolds. Concretely speaking, one may regard the 4-loop sunrise [quadruple integral in \eqref{eq:4-loop_sunrise}] as a motivic integral over the Barth--Nieto quintic variety \cite{BarthNieto1994,GritsenkoHulek1999,Hulek2001}, which is defined through a complete intersection\begin{align}
N:=\left\{[u_0:u_{1}:u_{2}:u_{3}:u_{4}:u_{5}]\in\mathbb P^5\left|\sum_{k=0}^5 u_k=\sum^5_{k=0}\frac{1}{u_k}=0\right.\right\}.\label{eq:BN}
\end{align}The projective variety $N$ has a smooth Calabi--Yau model $Y$. Its third \'etale cohomology group $H^3_{\text{\'et}}(Y) $   is related to  2-dimensional representations of $\Gal(\overline{\mathbb Q}/\mathbb Q) $  \cite[\S3]{Hulek2001}, so that for each prime $p\geq5$, one has  $L_p(H^3_{\text{\'et}}(Y),s)=[1-a_p(Y)p^{-s}+p^{3-2s}]^{-1} $ for \begin{align} a_p(Y)=\tr(\Frob_p^*,H^3_{\text{\'et}}(Y))=1+50p+50p^{2}+p^3-\#Y(\mathbb F_p),\end{align}where $\#Y(\mathbb F_p) $ counts the number of points within   $ Y$ over the finite field $ \mathbb F_p$.
 The modular $L$-function $ L(f_{4,6},s)$ coincides with  $ L(H^3_{\text{\'et}}(Y),s)=\prod_p L_p(H^3_{\text{\'et}}(Y),s)$ for all  the local factors $L_p(\cdot,s) $ corresponding to primes $ p\geq5$ and $\RE s$ sufficiently large.
A similar $p$-adic reinterpretation for  $ L(f_{3,15},s)$ also exists. Let $A_n $ be the Fourier coefficient in  $ f_{3,15}(z)=\sum_{n=1}^\infty A_ne^{2\pi in z}$, and $ \left(\frac p3\right)$ be the Legendre symbol for a prime $p$ other than $3$ and $5$, then \cite[Theorem 5.3]{PetersTopVlugt1992}\begin{align}
1+p^2+p\left( 16+4\left(\frac p3\right) \right)+A_p
\end{align}counts the number of $ \mathbb F_p$-rational points of a $K3$ surface
that is the minimal resolution of singularities of \begin{align}
\left\{[u_0:u_{1}:u_{2}:u_{3}:u_{4}]\in\mathbb P^4\left|\sum_{k=0}^4 u_k=\sum^4_{k=0}\frac{1}{u_k}=0\right.\right\}.
\end{align}

Behind the aforementioned results on $p$-adic Bessel moments is a long and heroic tradition of algebraic geometry. Back in the 1970s, building upon the theories of Dwork \cite{Dwork1960} and Grothendieck \cite{Grothendieck1968,SGA5}, Deligne  interpreted  Hasse--Weil $ L$-functions as Fredholm determinants of Frobenius maps \cite[(1.5.4)]{Deligne1Weil}.
This tradition has been continued by  Robba \cite{Robba1986pBessel}, Fu--Wan \cite{FuWan2005,FuWan2008,FuWan2008MathAnn} and Yun \cite{Yun2015}, in their studies of $p$-adic Bessel functions and Kloosterman sheaves.

While Broadhurst's $p$-adic heuristics give strong hints that  $ L(f_{3,15},s)$, $ L(f_{4,6},s)$ and $ L(f_{6,6},s)$ are appropriate mathematical models for 5-, 6- and 8-Bessel problems, our proofs of Theorems \ref{thm:Bologna} and \ref{thm:specL} described in this survey do not touch upon the  $p$-adic structure. It is perhaps worthwhile to rework these proofs from the Hasse--Weil perspective, using local-global correspondence. We call for this effort because there are still many conjectures of Broadhurst  (see \S\ref{subsec:open_prob} for a partial list) that go beyond the reach of this survey, but might appear tractable to specialists in $p$-adic analysis and \'etale cohomology.
\subsection{Open questions\label{subsec:open_prob}}There are three outstanding problems involving  5-, 6- and 8-Bessel factors, originally formulated by Broadhurst--Mellit \cite[(4.3), (5.8), (7.15)]{BroadhurstMellit2016} and Broadhurst \cite[(101), (114),  (160)]{Broadhurst2016}. \begin{conjecture}[Broadhurst--Mellit]\label{conj:568}The following determinant formulae hold:
\begin{align}
\det
\begin{pmatrix}\IKM(0,5;1) & \IKM(0,5;3) \\
\IKM(2,3;1) & \IKM(2,3;3) \\
\end{pmatrix}\overset{?}{=}{}&\frac{45}{8\pi^{2}}L(f_{3,15},4),
\\\det
\begin{pmatrix}\IKM(0,6;1) & \IKM(0,6;3) \\
\IKM(2,4;1) & \IKM(2,4;3) \\
\end{pmatrix}\overset{?}{=}{}&\frac{27}{4\pi^2}L(f_{4,6},5),\\\det
\begin{pmatrix}\IKM(0,8;1) & \IKM(0,8;3)-2\IKM(0,8;5) \\
\IKM(2,6;1) & \IKM(2,6;3)-2\IKM(2,6;5) \\
\end{pmatrix}\overset{?}{=}{}&\frac{6075}{128\pi^2}L(f_{6,6},7).\end{align}\end{conjecture}

Here, one might wish to compare the last conjectural determinant evaluation to the following proven result:\begin{align}\begin{split}
\frac{5\pi^8}{2^{19}3}={}&\det\mathbf  N_3=\det\begin{pmatrix}\IKM (1,7;1)& \IKM (1,7;3) & \IKM (1,7;5) \\
\IKM (2,6;1)& \IKM (2,6;3) & \IKM (2,6;5) \\
\IKM (3,5;1)& \IKM (3,5;3) & \IKM (3,5;5) \\
\end{pmatrix}\\={}&\frac{\pi^{2}}{2^{8}}\det\begin{pmatrix}\IKM (1,7;1)& \IKM (1,7;3)-2\IKM (1,7;5)  \\
\IKM (2,6;1)& \IKM (2,6;3)-2\IKM (2,6;5) \\
\end{pmatrix}.\end{split}
\end{align}To arrive at the  last step, we have used the Crandall number relations [Theorem \ref{thm:B3G_Crandall}\ref{itm:Crandall}]  $ \IKM (3,5;1)-\frac{\IKM (1,7;1)}{\pi^{2}}=0$, $ \IKM (3,5;3)-\frac{\IKM (1,7;3)}{\pi^{2}}=\frac{\pi^{2}}{2^{7}}$, $ \IKM (3,5;5)-\frac{\IKM (1,7;5)}{\pi^{2}}=\frac{\pi^{2}}{2^{8}}$, along with row and column eliminations.

The special $ L$-values $L(f_{k,N},s) $ in Conjecture \ref{conj:568} all lie outside the critical strip $ 0<\RE s< k$, so they do not yield to the methods given in \S\ref{subsec:mod_form} or \S\ref{sec:L-value}.

Working with Anton Mellit at Mainz, David Broadhurst has discovered a numerical connection (see \cite[(6.8)]{BroadhurstMellit2016} or \cite[(129)]{Broadhurst2016}) between $ \zeta_{7,1}(s)=\prod_p \frac{1}{Z_7(p,p^{-s})}$ and the 7-Bessel problem, which still awaits a proof.

\begin{conjecture}[Broadhurst--Mellit]We have \begin{align}\IKM(2,5;1)\overset{?}{=}\frac{5\pi^2}{24}\zeta_{7,1}(2).\end{align}\end{conjecture}

In a recent collaboration with David Roberts \cite{Broadhurst2017Paris,Broadhurst2017CIRM,Broadhurst2017Higgs,Broadhurst2017ESIa,Broadhurst2017DESY}, David Broadhurst has discovered a lot more empirical formulae relating determinants of Bessel moments to special values of Hasse--Weil $L$-functions, which are outside the scope of the current exposition.   Nevertheless, we  believe that  one day such determinant formulae will reveal deep $p$-adic structures of Bessel moments, as foreshadowed by pioneering works on Hasse--Weil $ L$-functions and Fredholm determinants for Frobenius maps \cite{Deligne1Weil,Robba1986pBessel,PetersTopVlugt1992,Hulek2001,FuWan2005,FuWan2008,FuWan2008MathAnn,Yun2015}.  \begin{acknowledgement}
This research was supported in part  by the Applied Mathematics Program within the Department of Energy
(DOE) Office of Advanced Scientific Computing Research (ASCR) as part of the Collaboratory on
Mathematics for Mesoscopic Modeling of Materials (CM4).

My work on Bessel moments and modular forms began in 2012, in the form of preliminary research notes at Princeton. I thank Prof.~Weinan E\ (Princeton University and Peking University) for running seminars on mathematical problems in quantum field theory at Princeton, and for arranging my stays at both Princeton and Beijing.

I am grateful to Dr.\ David Broadhurst for fruitful communications on recent progress in the arithmetic studies of Feynman diagrams  \cite{Broadhurst2017Paris,Broadhurst2017CIRM,Broadhurst2017Higgs,Broadhurst2017ESIa,Broadhurst2017DESY}. It is a pleasure to dedicate this survey to him, in honor of his 70th birthday.
\end{acknowledgement}


\clearpage
\setcounter{page}{1}
\section*{Correction to: Some algebraic and arithmetic properties of
Feynman diagrams}

Form the  Chan--Zudilin groups  $ \Gamma_0(6)_{+6}$,  $ \Gamma_0(6)_{+3}$, and  $ \Gamma_0(6)_{+2}$ by adjoining the Chan--Zudilin involutions $ \hat W_6z= -1/(6z)$, $ \hat W_3z= (3z-2)/(6z-3)$, and $ \hat W_2z= (2z-1)/(6z-2)$ to $ \Gamma_0(6)$, the Hecke congruence subgroup of level 6.

Panels \textit{d} and \textit{f}  in \cite[Fig.~2]{Zhou2018ExpoDESY}, which originally portrayed $\mathfrak D_{6,2} $ and $\mathfrak D_{6,6} $ [the fundamental domains of $ \Gamma_0(6)_{+2}$  and  $ \Gamma_0(6)_{+6}$] in the following form:

\vspace{-1.5em}
\begin{figure*}[h!]
\begin{center}

\begin{tikzpicture}[scale=2]

\foreach \y in {0}\draw[->]({\y},-.25)--({\y},1.625)node[above]{$\IM z$};
\foreach \y in {0} \draw[color=gray!20!white,fill,ultra thin] ({\y+1/2},1.5)--({\y+1/2},0)--plot[samples=360,domain=0:180] ({\y+5/12+cos(\x)/12},{sin(\x)/12})--plot[samples=360,domain=0:180] ({\y+1/6+cos(\x)/6},{sin(\x)/6})--({\y},{0})--({\y},1.5);

\foreach \y in {0}
\draw[->]({\y-0.1},0)--({\y+0.6},0)node[right]{$\RE z$};

\draw[->,color=red,line width=1pt] (.5,0)--(.5,.8);
\draw[color=red,line width=1pt] (.5,0)--(.5,1.5);
\draw[->,color=blue,line width=.5pt] plot[samples=20,domain=100:90] ({cos(\x)/12+5/12},{sin(\x)/12});
\draw[->,color=blue,line width=.5pt] plot[samples=20,domain=80:70] ({cos(\x)/12+5/12},{sin(\x)/12});
\draw[color=blue,line width=1pt] plot[samples=360,domain=0:180] ({cos(\x)/12+5/12},{sin(\x)/12});
\draw[color=blue,line width=1pt] (0,0)--(0,1.5);
\draw[->,color=red,line width=.5pt] plot[samples=20,domain=100:90] ({cos(\x)/6+1/6},{sin(\x)/6});
\draw[color=red,line width=1pt] plot[samples=360,domain=0:180] ({cos(\x)/6+1/6},{sin(\x)/6});
\draw[>-,color=blue,line width=1pt] (0,0.8)--(0,0);
\draw[>-,color=blue,line width=1pt] (0,0.75)--(0,0);

\draw({1/2},0)node[below]{$^{\frac12} $};
\draw({1/3},0)node[below]{$^{\frac13} $};

\node[color=black] at (0.375,1.375) {$^{ \mathfrak D_{6,2}}$};
\draw(0.25,-0.25) node[below]{\textbf{(d$ ^\times$)}};

\end{tikzpicture}\hspace{.12em}\begin{tikzpicture}[scale=2]

\foreach \y in {0}\draw[->]({\y},-.25)--({\y},1.625)node[above]{$\IM z$};
\foreach \y in {0} \draw[color=gray!20!white,fill,ultra thin] ({\y+1/2},1.5)--({\y+1/2},0)-- plot[samples=200,domain=0:101.5] ({5/12+cos(\x)/12},{sin(\x)/12})--plot[samples=360,domain=11.5:168.5] ({cos(\x)/sqrt(6)},{sin(\x)/sqrt(6)})-- plot[samples=200,domain=78.5:180] ({-5/12+cos(\x)/12},{sin(\x)/12})--({\y-1/2},0)--({\y-1/2},1.5);

\foreach \y in {0}\draw[->]({\y-0.6},0)--({\y+0.6},0)node[right]{$\RE z$};

\draw[->,color=green,line width=1pt] (.5,0)--(.5,.8);
\draw[->,color=green,line width=1pt] (-.5,0)--(-.5,.8);
\draw[->,color=green,line width=1pt] (.5,0)--(.5,.75);
\draw[->,color=green,line width=1pt] (-.5,0)--(-.5,.75);
\draw[->,color=green,line width=1pt] (.5,0)--(.5,.7);
\draw[->,color=green,line width=1pt] (-.5,0)--(-.5,.7);
\draw[color=green,line width=1pt] (.5,0)--(.5,1.5);
\draw[color=green,line width=1pt] (-.5,0)--(-.5,1.5);

\draw[dotted,thick] plot[samples=180,domain=90:180] ({cos(\x)/sqrt(6)},{sin(\x)/sqrt(6)});
\draw[dotted,thick] plot[samples=180,domain=0:90] ({cos(\x)/sqrt(6)},{sin(\x)/sqrt(6)});
\draw[dotted,thick] plot[samples=360,domain=0:180] ({5/12+cos(\x)/12},{sin(\x)/12});
\draw[dotted,thick] plot[samples=360,domain=0:180] ({-5/12+cos(\x)/12},{sin(\x)/12});
\draw[->,color=blue,line width=.5pt] plot[samples=20,domain=130:140] ({-5/12+cos(\x)/12},{sin(\x)/12});
\draw[->,color=blue,line width=.5pt] plot[samples=20,domain=110:120] ({-5/12+cos(\x)/12},{sin(\x)/12});
\draw[->,color=blue,line width=.1pt] plot[samples=20,domain=70:60] ({5/12+cos(\x)/12},{sin(\x)/12});
\draw[->,color=blue,line width=.1pt] plot[samples=20,domain=50:40] ({5/12+cos(\x)/12},{sin(\x)/12});
\draw[color=blue,line width=1pt] plot[samples=200,domain=0:101.5] ({5/12+cos(\x)/12},{sin(\x)/12});
\draw[color=blue,line width=1pt] plot[samples=200,domain=78.5:180] ({-5/12+cos(\x)/12},{sin(\x)/12});
\draw[->,color=red,line width=.5pt] plot[samples=20,domain=120:130] ({cos(\x)/sqrt(6)},{sin(\x)/sqrt(6)});
\draw[->,color=red,line width=.5pt] plot[samples=20,domain=60:50] ({cos(\x)/sqrt(6)},{sin(\x)/sqrt(6)});
\draw[color=red,line width=1pt] plot[samples=360,domain=11.5:168.5] ({cos(\x)/sqrt(6)},{sin(\x)/sqrt(6)});

\draw({1/2},0)node[below]{$^{\frac12} $};
\draw({1/3},0)node[below]{$^{\frac13} $};

\draw[fill=black](0,{1/sqrt(6)})circle[radius=0.015]node at (0,{1/sqrt(6)+.125}) {$_{\frac{i}{\sqrt6}}$};
\node[color=black] at (0.375,1.375) {$^{ \mathfrak D_{6,6}}$};
\draw(0,-0.25) node[below]{\textbf{(f$ ^\circ$)}};\end{tikzpicture}
\end{center}\end{figure*}
\vspace{-2em}
\noindent need to be replaced by

\vspace{-1.5em}
\begin{figure*}[h!]\begin{center}

\begin{tikzpicture}[scale=2]

\foreach \y in {0}\draw[->]({\y},-.25)--({\y},1.625)node[above]{$\IM z$};
\foreach \y in {0} \draw[color=gray!20!white,fill,ultra thin] ({\y+1/2},1.5)--({\y+1/2},0)--plot[samples=360,domain=0:180] ({\y+5/12+cos(\x)/12},{sin(\x)/12})--plot[samples=360,domain=0:180] ({\y+1/6+cos(\x)/6},{sin(\x)/6})--({\y},{0})--({\y},1.5);

\foreach \y in {0}
\draw[->]({\y-0.15},0)--({\y+0.6},0)node[right]{$\RE z$};

\draw[->,color=green,line width=1pt] (.5,0)--(.5,.8);
\draw[->,color=green,line width=1pt] (.5,0)--(.5,.75);
\draw[color=green,line width=1pt] (.5,0)--(.5,1.5);
\draw[->,color=blue,line width=.1pt] plot[samples=20,domain=60:50] ({5/12+cos(\x)/12},{sin(\x)/12});
\draw[->,color=blue,line width=.1pt] plot[samples=20,domain=140:150] ({5/12+cos(\x)/12},{sin(\x)/12});
\draw[color=blue,line width=1pt] plot[samples=360,domain=0:180] ({cos(\x)/12+5/12},{sin(\x)/12});
\draw[color=brown,line width=1pt] (0,0)--(0,1.5);
\draw[->,color=green,line width=.5pt] plot[samples=20,domain=80:90] ({cos(\x)/6+1/6},{sin(\x)/6});
\draw[->,color=green,line width=.5pt] plot[samples=20,domain=90:100] ({cos(\x)/6+1/6},{sin(\x)/6});
\draw[color=green,line width=1pt] plot[samples=360,domain=0:180] ({cos(\x)/6+1/6},{sin(\x)/6});
\draw[>-,color=brown,line width=1pt] (0,0.8)--(0,0);
\draw[>-,color=brown,line width=1pt] (0,0.75)--(0,0);
\draw[>-,color=brown,line width=1pt] (0,0.7)--(0,0);
\draw[>-,color=brown,line width=1pt] (0,0.2)--(0,0.5);
\draw[>-,color=brown,line width=1pt] (0,0.25)--(0,.5);
\draw[>-,color=brown,line width=1pt] (0,0.15)--(0,.5);

\draw[dotted,thick] plot[samples=180,domain=0:90] ({cos(\x)/sqrt(6)},{sin(\x)/sqrt(6)});

\draw({1/2},0)node[below]{$^{\frac12} $};
\draw({1/3},0)node[below]{$^{\frac13} $};
\draw[fill=black](0,{1/sqrt(6)})circle[radius=0.02]node at (-.1,{1/sqrt(6)-.02}) {$_{\frac{i}{\sqrt6}}$};
\draw[fill=black]({2/5},{1/5/sqrt(6)})circle[radius=0.02];

\node[color=black] at (0.375,1.375) {$^{ \mathfrak D_{6,6}}$};
\draw(0.25,-0.25) node[below]{\textbf{(d)}};

\end{tikzpicture}\begin{tikzpicture}[scale=2]

\foreach \y in {0}\draw[->]({\y},-.25)--({\y},1.625)node[above]{$\IM z$};
\foreach \y in {0} \draw[color=gray!20!white,fill,ultra thin] ({\y+1/2},1.5)--({\y+1/2},0)-- plot[samples=200,domain=0:101.5] ({5/12+cos(\x)/12},{sin(\x)/12})--plot[samples=360,domain=11.5:168.5] ({cos(\x)/sqrt(6)},{sin(\x)/sqrt(6)})-- plot[samples=200,domain=78.5:180] ({-5/12+cos(\x)/12},{sin(\x)/12})--({\y-1/2},0)--({\y-1/2},1.5);

\foreach \y in {0}\draw[->]({\y-0.6},0)--({\y+0.6},0)node[right]{$\RE z$};

\draw[->,color=green,line width=1pt] (.5,0)--(.5,.8);
\draw[->,color=green,line width=1pt] (-.5,0)--(-.5,.8);
\draw[->,color=green,line width=1pt] (.5,0)--(.5,.75);
\draw[->,color=green,line width=1pt] (-.5,0)--(-.5,.75);
\draw[->,color=green,line width=1pt] (.5,0)--(.5,.7);
\draw[->,color=green,line width=1pt] (-.5,0)--(-.5,.7);
\draw[color=green,line width=1pt] (.5,0)--(.5,1.5);
\draw[color=green,line width=1pt] (-.5,0)--(-.5,1.5);

\draw[dotted,thick] plot[samples=180,domain=90:180] ({cos(\x)/sqrt(6)},{sin(\x)/sqrt(6)});
\draw[dotted,thick] plot[samples=180,domain=0:90] ({cos(\x)/sqrt(6)},{sin(\x)/sqrt(6)});
\draw[dotted,thick] plot[samples=360,domain=0:180] ({5/12+cos(\x)/12},{sin(\x)/12});
\draw[dotted,thick] plot[samples=360,domain=0:180] ({-5/12+cos(\x)/12},{sin(\x)/12});
\draw[->,color=blue,line width=.5pt] plot[samples=20,domain=130:140] ({-5/12+cos(\x)/12},{sin(\x)/12});
\draw[->,color=blue,line width=.5pt] plot[samples=20,domain=110:120] ({-5/12+cos(\x)/12},{sin(\x)/12});
\draw[->,color=blue,line width=.1pt] plot[samples=20,domain=70:60] ({5/12+cos(\x)/12},{sin(\x)/12});
\draw[->,color=blue,line width=.1pt] plot[samples=20,domain=50:40] ({5/12+cos(\x)/12},{sin(\x)/12});
\draw[color=blue,line width=1pt] plot[samples=200,domain=0:101.5] ({5/12+cos(\x)/12},{sin(\x)/12});
\draw[color=blue,line width=1pt] plot[samples=200,domain=78.5:180] ({-5/12+cos(\x)/12},{sin(\x)/12});
\draw[->,color=magenta,line width=.5pt] plot[samples=20,domain=130:120] ({cos(\x)/sqrt(6)},{sin(\x)/sqrt(6)});
\draw[->,color=magenta,line width=.5pt] plot[samples=20,domain=150:160] ({cos(\x)/sqrt(6)},{sin(\x)/sqrt(6)});
\draw[->,color=yellow,line width=.5pt] plot[samples=20,domain=30:20] ({cos(\x)/sqrt(6)},{sin(\x)/sqrt(6)});
\draw[->,color=yellow,line width=.5pt] plot[samples=20,domain=50:60] ({cos(\x)/sqrt(6)},{sin(\x)/sqrt(6)});
\draw[color=yellow,line width=1pt] plot[samples=360,domain=11.5:90] ({cos(\x)/sqrt(6)},{sin(\x)/sqrt(6)});
\draw[color=magenta,line width=1pt] plot[samples=360,domain=90:168.5] ({cos(\x)/sqrt(6)},{sin(\x)/sqrt(6)});

\draw[dotted,thick] ({1/3},0)--({1/3},{sqrt(2)/6});
\draw[dotted,thick] ({-1/3},0)--({-1/3},{sqrt(2)/6});
\draw[fill=yellow,draw=yellow]({1/3},{sqrt(2)/6})circle[radius=0.02];
\draw[fill=magenta,draw=magenta]({-1/3},{sqrt(2)/6})circle[radius=0.02];

\draw({1/2},0)node[below]{$^{\frac12} $};
\draw({1/3},0)node[below]{$^{\frac13} $};

\draw[fill=black](0,{1/sqrt(6)})circle[radius=0.02]node at (0,{1/sqrt(6)+.125}) {$_{\frac{i}{\sqrt6}}$};
\node[color=black] at (0.375,1.375) {$^{ \mathfrak D_{6,2}}$};
\draw(0,-0.25) node[below]{\textbf{(f)}};
\end{tikzpicture}\end{center}
\end{figure*}
\vspace{-2em}
\noindent with the accompanying figure caption intact. In association with the update in panel \textit{d}, the boundary of the plotted region should  be now referred to as $ \partial \mathfrak D_{6,6}$ in a sentence that immediately follows \cite[(83)]{Zhou2018ExpoDESY}.

Here, we point out that  panel  \textit{f}$ ^\circ$ serves as an alternative design for $ \mathfrak D_{6,6}$, the fundamental domain of the Chan--Zudilin group $ \Gamma_0(6)_{+6}$, but  panel \textit{d}$ ^\times$ does not qualify as a fundamental domain of $ \Gamma_0(6)_{+2}$ (which is to be visualized in the next paragraph).

Recall the modular invariants   $ X_{6,6}(z)=[\eta(z)\eta(6z)]^{12}/[\eta(2z)\eta(3z)]^{12}$, $ X_{6,3}(z)=[\eta(2z)\eta(6z)]^6/[\eta(z)\eta(3z)]^6$, and  $ X_{6,2}(z)=[\eta(3z)\eta(6z)]^4/[\eta(z)\eta(2z)]^4$ from \cite[(2.1)--(2.3)]{ChanZudilin2010a}, where   $ \eta(z)=e^{\pi iz/12}\prod_{n=1}^\infty(1-e^{2\pi inz})$ is the Dedekind eta function for $ \IM z>0$. We subdivide (any design of)  the fundamental domain $ \mathfrak D_{6,k}$ for the Chan--Zudilin group $\Gamma_0(6)_{+k}$ ($ k\in\{6,3,2\}$) into  $ \mathfrak D_{6,k}^+$ and $  \mathfrak D_{6,k}^- $, according to the sign of $ \IM X_{6,k}(z)$. Using modular transformations in $ \Gamma_0(6)_{+k}$, we can fit two copies of these sub-divided fundamental domains into $ \mathfrak D_6$, the fundamental domain of $ \Gamma_0(6)$ \cite[Fig.~2\textit{c}]{Zhou2018ExpoDESY}. For the current version of panels \textit{d}--\textit{f}, such  two-fold correspondences (for $ [\Gamma_{0}(6)_{+k}:\Gamma_{0}(6)]=2$) are illustrated below:

\vspace{-1.5em}
\begin{figure*}[h!]
\begin{center}

\begin{tikzpicture}[scale=2.65]

\foreach \y in {0}\draw[->]({\y},-.25)--({\y},.625)node[above]{$\IM z$};
\foreach \y in {0} \draw[color=gray!20!white,fill,ultra thin] ({\y+1/2},0.5)--({\y+1/2},0)-- plot[samples=200,domain=0:101.5] ({5/12+cos(\x)/12},{sin(\x)/12})--plot[samples=360,domain=11.5:90] ({cos(\x)/sqrt(6)},{sin(\x)/sqrt(6)})-- ({\y},0.5);
\foreach \y in {0} \draw[color=gray!20!white,fill,ultra thin] ({\y},{1/sqrt(6)})--({\y},0)-- plot[samples=200,domain=0:180] ({-1/6+cos(\x)/6},{sin(\x)/6})--plot[samples=360,domain=0:79.5] ({-5/12+cos(\x)/12},{sin(\x)/12})--plot[samples=360,domain={180-11.5}:90] ({cos(\x)/sqrt(6)},{sin(\x)/sqrt(6)})-- ({\y},{1/sqrt(6)});

\foreach \y in {0}\draw[->]({\y-0.6},0)--({\y+0.6},0)node[right]{$\RE z$};

\draw[->,color=green,line width=1pt] (.5,0)--(.5,.4);
\draw[->,color=green,line width=1pt] (-.5,0)--(-.5,.4);
\draw[color=brown,line width=1pt] (0,0)--(0,.5);

\draw[->,color=brown,line width=.5pt] (0,.45)--(0,.44);
\draw[->,color=brown,line width=.5pt] (0,.25)--(0,.3);
\draw[color=green,line width=1pt] (.5,0)--(.5,.5);
\draw[color=green,line width=1pt] (-.5,0)--(-.5,.5);

\draw[dotted,thick] plot[samples=180,domain=90:180] ({cos(\x)/sqrt(6)},{sin(\x)/sqrt(6)});
\draw[dotted,thick] plot[samples=180,domain=0:90] ({cos(\x)/sqrt(6)},{sin(\x)/sqrt(6)});
\draw[color=green,line width=1pt] plot[samples=180,domain=0:180] ({1/6+cos(\x)/6},{sin(\x)/6});

\draw[color=green,line width=1pt] plot[samples=180,domain=0:180] ({-1/6+cos(\x)/6},{sin(\x)/6});

\draw[->,color=green,line width=1pt] plot[samples=20,domain=100:110] ({1/6+cos(\x)/6},{sin(\x)/6});
\draw[->,color=green,line width=1pt] plot[samples=20,domain=80:70] ({-1/6+cos(\x)/6},{sin(\x)/6});

\draw[dotted,thick] plot[samples=360,domain=0:180] ({5/12+cos(\x)/12},{sin(\x)/12});
\draw[dotted,thick] plot[samples=360,domain=0:180] ({-5/12+cos(\x)/12},{sin(\x)/12});
\draw[->,color=blue,line width=.5pt] plot[samples=20,domain=40:30] ({-5/12+cos(\x)/12},{sin(\x)/12});
\draw[->,color=blue,line width=.5pt] plot[samples=20,domain=115:125] ({-5/12+cos(\x)/12},{sin(\x)/12});
\draw[->,color=blue,line width=.5pt] plot[samples=20,domain=70:60] ({5/12+cos(\x)/12},{sin(\x)/12});
\draw[->,color=blue,line width=.5pt] plot[samples=20,domain=140:150] ({5/12+cos(\x)/12},{sin(\x)/12});
\draw[color=blue,line width=1pt] plot[samples=200,domain=0:180] ({5/12+cos(\x)/12},{sin(\x)/12});
\draw[color=blue,line width=1pt] plot[samples=200,domain=0:180] ({-5/12+cos(\x)/12},{sin(\x)/12});
\draw[->,color=magenta,line width=.5pt] plot[samples=20,domain=120:130] ({cos(\x)/sqrt(6)},{sin(\x)/sqrt(6)});
\draw[->,color=magenta,line width=.5pt] plot[samples=20,domain=60:50] ({cos(\x)/sqrt(6)},{sin(\x)/sqrt(6)});
\draw[color=magenta,line width=1pt] plot[samples=360,domain=11.5:168.5] ({cos(\x)/sqrt(6)},{sin(\x)/sqrt(6)});

\draw({1/2},0)node[below]{$^{\frac12} $};
\draw({1/3},0)node[below]{$^{\frac13} $};

\node[color=black] at (-0.32,.42) {$^{\hat W_6 \mathfrak D_{6,6}^-}$};

\node[color=black] at (0.11,.21) {$^{ \mathfrak D_{6,6}^-}$};
\node[color=black] at (-0.15,.21) {$^{ \hat W_6\mathfrak D_{6,6}^+}$};
\node[color=black] at (0.375,.42) {$^{ \mathfrak D_{6,6}^+}$};
\draw(0,-0.25) node[below]{{\textbf{(d$'$)}}};

\end{tikzpicture}\begin{tikzpicture}[scale=2.65]

\foreach \y in {0}\draw[->]({\y},-.25)--({\y},.625)node[above]{$\IM z$};

\foreach \y in {0} \draw[color=gray!20!white,fill,ultra thin] ({\y+1/2},.5)--({\y+1/2},{.5/sqrt(3)})--plot[samples=360,domain=90:150] ({.5+.5/sqrt(3)*cos(\x)},{.5/sqrt(3)*sin(\x)})--plot[samples=360,domain=60:180] ({\y+1/6+cos(\x)/6},{sin(\x)/6})--({\y},{0})--({\y},.5);

\foreach \y in {0} \draw[color=gray!20!white,fill,ultra thin] (-.5,{.5/sqrt(3)})--(-.5,0)--plot[samples=360,domain=180:0] ({-5/12+cos(\x)/12},{sin(\x)/12})--plot[samples=360,domain=180:120] ({\y-1/6+cos(\x)/6},{sin(\x)/6})--plot[samples=360,domain=30:90] ({-1/2+cos(\x)/2/sqrt(3)},{sin(\x)/2/sqrt(3)});

\foreach \y in {0}\draw[->]({\y-0.6},0)--({\y+0.6},0)node[right]{$\RE z$};

\draw[dotted,thick] plot[samples=180,domain=90:180] ({.5+.5/sqrt(3)*cos(\x)},{.5/sqrt(3)*sin(\x)});
\draw[dotted,thick] plot[samples=180,domain=0:90] ({-.5+.5/sqrt(3)*cos(\x)},{.5/sqrt(3)*sin(\x)});
\draw[dotted,thick] plot[samples=360,domain=0:180] ({cos(\x)/6-1/6},{sin(\x)/6});
\draw[dotted,thick] plot[samples=360,domain=0:180] ({cos(\x)/6+1/6},{sin(\x)/6});

\draw[color=cyan,line width=1pt] plot[samples=120,domain=90:150] ({.5+.5/sqrt(3)*cos(\x)},{.5/sqrt(3)*sin(\x)});
\draw[color=cyan,line width=1pt] plot[samples=120,domain=30:90] ({-.5+.5/sqrt(3)*cos(\x)},{.5/sqrt(3)*sin(\x)});

\draw[->,color=green,line width=1pt] (.5,0)--(.5,.4);
\draw[->,color=green,line width=1pt] (-.5,0)--(-.5,.4);
\draw[->,color=green,line width=1pt] (.5,.2)--(.5,.15);
\draw[->,color=green,line width=1pt] (-.5,.2)--(-.5,.15);
\draw[color=green,line width=1pt] (.5,0)--(.5,.5);
\draw[color=green,line width=1pt] (-.5,0)--(-.5,.5);
\draw[color=blue,line width=1pt] (0,0)--(0,.5);
\draw[->,color=blue,line width=1pt] (0,0.3)--(0,.25);

\draw[->,color=blue,line width=.5pt] plot[samples=20,domain=90:100] ({5/12+cos(\x)/12},{sin(\x)/12});
\draw[->,color=blue,line width=.5pt] plot[samples=20,domain=100:90] ({-5/12+cos(\x)/12},{sin(\x)/12});
\draw[color=blue,line width=1pt] plot[samples=200,domain=0:180] ({5/12+cos(\x)/12},{sin(\x)/12});
\draw[color=blue,line width=1pt] plot[samples=200,domain=0:180] ({-5/12+cos(\x)/12},{sin(\x)/12});

\draw[->,color=cyan,line width=.5pt] plot[samples=20,domain=50:60] ({-.5+.5/sqrt(3)*cos(\x)},{.5/sqrt(3)*sin(\x)});
\draw[->,color=cyan,line width=.5pt] plot[samples=20,domain=130:120] ({.5+.5/sqrt(3)*cos(\x)},{.5/sqrt(3)*sin(\x)});
\draw[->,color=red,line width=.5pt] plot[samples=20,domain=90:100] ({cos(\x)/6-1/6},{sin(\x)/6});
\draw[->,color=red,line width=.5pt] plot[samples=20,domain=100:90] ({cos(\x)/6+1/6},{sin(\x)/6});
\draw[->,color=red,line width=.5pt] plot[samples=20,domain=150:140] ({cos(\x)/6-1/6},{sin(\x)/6});
\draw[->,color=red,line width=.5pt] plot[samples=20,domain=30:40] ({cos(\x)/6+1/6},{sin(\x)/6});

\draw[color=red,line width=1pt] plot[samples=240,domain=0:180] ({cos(\x)/6-1/6},{sin(\x)/6});
\draw[color=red,line width=1pt] plot[samples=240,domain=0:180] ({cos(\x)/6+1/6},{sin(\x)/6});

\draw({1/2},0)node[below]{$^{\frac12} $};
\draw({1/3},0)node[below]{$^{\frac13} $};
\node[color=black] at (-0.375,.42) {$^{ \mathfrak D_{6,3}^-}$};

\node[color=black] at (0.375,.42) {$^{\mathfrak D_{6,3}^+}$};
\draw[<-,black,text=black,ultra thin] (-{5/12},.2)--(-{5/12},-.175) node at (-.4,-.23){$_{\hat \tau\hat W_3\mathfrak D_{6,3}^+}$};
\draw[<-,black,text=black,ultra thin] ({5/12},.2)--({5/12},-.175) node at (.4,-.23){$_{\hat W_3\hat T\mathfrak D_{6,3}^-}$};
\draw(0,-0.25) node[below]{{\textbf{(e$'$)}}};
\end{tikzpicture}\begin{tikzpicture}[scale=2.65]

\foreach \y in {0}\draw[->]({\y},-.25)--({\y},.625)node[above]{$\IM z$};
\foreach \y in {0} \draw[color=gray!20!white,fill,ultra thin] ({\y+1/2},0.5)--({\y+1/2},0)-- plot[samples=200,domain=0:101.5] ({5/12+cos(\x)/12},{sin(\x)/12})--plot[samples=360,domain=11.5:90] ({cos(\x)/sqrt(6)},{sin(\x)/sqrt(6)})-- ({\y},0.5);
\foreach \y in {0} \draw[color=gray!20!white,fill,ultra thin] ({\y},{1/sqrt(6)})--({\y},0)-- plot[samples=200,domain=180:0] ({1/6+cos(\x)/6},{sin(\x)/6})--plot[samples=360,domain=180:{180-79.5}] ({5/12+cos(\x)/12},{sin(\x)/12})--plot[samples=360,domain=11.5:90] ({cos(\x)/sqrt(6)},{sin(\x)/sqrt(6)})-- ({\y},{1/sqrt(6)});

\foreach \y in {0}\draw[->]({\y-0.6},0)--({\y+0.6},0)node[right]{$\RE z$};

\draw[->,color=green,line width=1pt] (.5,0)--(.5,.4);
\draw[->,color=green,line width=1pt] (-.5,0)--(-.5,.4);
\draw[color=blue,line width=1pt] (0,0)--(0,{1/sqrt(6)});
\draw[color=cyan,line width=1pt] (0,0.5)--(0,{1/sqrt(6)});

\draw[-<,color=cyan,line width=.5pt] (0,.45)--(0,.44);
\draw[->,color=blue,line width=.5pt] (0,.3)--(0,.25);
\draw[color=green,line width=1pt] (.5,0)--(.5,.5);
\draw[color=green,line width=1pt] (-.5,0)--(-.5,.5);

\draw[dotted,thick] plot[samples=180,domain=90:180] ({cos(\x)/sqrt(6)},{sin(\x)/sqrt(6)});
\draw[dotted,thick] plot[samples=180,domain=0:90] ({cos(\x)/sqrt(6)},{sin(\x)/sqrt(6)});
\draw[color=green,line width=1pt] plot[samples=180,domain=0:180] ({1/6+cos(\x)/6},{sin(\x)/6});

\draw[color=green,line width=1pt] plot[samples=180,domain=0:180] ({-1/6+cos(\x)/6},{sin(\x)/6});

\draw[->,color=green,line width=1pt] plot[samples=20,domain=110:100] ({1/6+cos(\x)/6},{sin(\x)/6});
\draw[->,color=green,line width=1pt] plot[samples=20,domain=70:80] ({-1/6+cos(\x)/6},{sin(\x)/6});

\draw[dotted,thick] plot[samples=360,domain=0:180] ({5/12+cos(\x)/12},{sin(\x)/12});
\draw[dotted,thick] plot[samples=360,domain=0:180] ({-5/12+cos(\x)/12},{sin(\x)/12});
\draw[->,color=cyan,line width=.5pt] plot[samples=20,domain=40:30] ({-5/12+cos(\x)/12},{sin(\x)/12});
\draw[->,color=blue,line width=.5pt] plot[samples=20,domain=115:125] ({-5/12+cos(\x)/12},{sin(\x)/12});
\draw[->,color=blue,line width=.5pt] plot[samples=20,domain=70:60] ({5/12+cos(\x)/12},{sin(\x)/12});
\draw[->,color=cyan,line width=.5pt] plot[samples=20,domain=140:150] ({5/12+cos(\x)/12},{sin(\x)/12});
\draw[color=blue,line width=1pt] plot[samples=200,domain=0:101.5] ({5/12+cos(\x)/12},{sin(\x)/12});
\draw[color=blue,line width=1pt] plot[samples=200,domain=78.5:180] ({-5/12+cos(\x)/12},{sin(\x)/12});

\draw[color=cyan,line width=1pt] plot[samples=200,domain=101.5:180] ({5/12+cos(\x)/12},{sin(\x)/12});
\draw[color=cyan,line width=1pt] plot[samples=200,domain=0:78.5] ({-5/12+cos(\x)/12},{sin(\x)/12});
\draw[->,color=magenta,line width=.5pt] plot[samples=20,domain=130:120] ({cos(\x)/sqrt(6)},{sin(\x)/sqrt(6)});
\draw[->,color=magenta,line width=.5pt] plot[samples=20,domain=150:160] ({cos(\x)/sqrt(6)},{sin(\x)/sqrt(6)});
\draw[->,color=yellow,line width=.5pt] plot[samples=20,domain=30:20] ({cos(\x)/sqrt(6)},{sin(\x)/sqrt(6)});
\draw[->,color=yellow,line width=.5pt] plot[samples=20,domain=50:60] ({cos(\x)/sqrt(6)},{sin(\x)/sqrt(6)});
\draw[color=yellow,line width=1pt] plot[samples=360,domain=11.5:90] ({cos(\x)/sqrt(6)},{sin(\x)/sqrt(6)});
\draw[color=magenta,line width=1pt] plot[samples=360,domain=90:168.5] ({cos(\x)/sqrt(6)},{sin(\x)/sqrt(6)});

\draw[dotted,thick] ({1/3},0)--({1/3},{sqrt(2)/6});
\draw[dotted,thick] ({-1/3},0)--({-1/3},{sqrt(2)/6});
\draw[fill=yellow,draw=yellow]({1/3},{sqrt(2)/6})circle[radius=0.02];
\draw[fill=magenta,draw=magenta]({-1/3},{sqrt(2)/6})circle[radius=0.02];
\draw({1/2},0)node[below]{$^{\frac12} $};
\draw({1/3},0)node[below]{$^{\frac13} $};

\node[color=black] at (-0.375,.42) {$^{\mathfrak D_{6,2}^-}$};

\node[color=black] at (0.15,.21) {$^{\hat W_2 \mathfrak D_{6,2}^+}$};
\node[color=black] at (-0.15,.21) {$^{ \hat W_2\hat\gamma\mathfrak D_{6,2}^-}$};
\node[color=black] at (0.375,.42) {$^{ \mathfrak D_{6,2}^+}$};
\draw(0,-0.25) node[below]{{\textbf{(f$'$)}}};

\end{tikzpicture}

\end{center}\end{figure*}

\vspace{-2em}
\noindent where $ \hat T z=z+1,\hat {\tau} z=z-1$ are horizontal translates, and  $ \hat\gamma z=(5z+2)/(12z+5)$ is a  $ \Gamma_0(6)$-transformation. In the graphical illustrations above, the subdivisions $ \mathfrak D_{6,k}^+$ and their $ \Gamma_0(6)_{+k}$-images are painted in gray. It is clear from panel \textit{f}$ '$ that the restriction of  $ X_{6,2}(z)$  to panel \textit{d}$^\times $ amounts to a two-fold cover of the upper half-plane, instead of the entire complex plane.

Originally, all the panels in \cite[Fig.~2]{Zhou2018ExpoDESY} were prepared by plotting the loci of real-valued modular invariants and identifying  equivalent sides on the loci. Such an approach became problematic for $ \Gamma_0(6)_{+2}$ (and hence panel \textit{d}$^\times$): it failed to detect the $ \Gamma_0(6)_{+2}$-equivalent sides that lie in the interior of $ \mathfrak D_{6,2}^+\cup\hat W_2\mathfrak D_{6,2}^+$ (see panel \textit{f}$ '$ above), on which $ \IM X_{6,2}(z)\neq0$.

After disqualifying panel \textit{d}$^\times $ from being a fundamental domain for $ \Gamma_0(6)_{+2}$, we continue with three more alternative formulations for the fundamental domains of the Chan--Zudilin groups, in the illustrations below.

\vspace{-1.5em}
\begin{figure*}[h!]
\begin{center}

\begin{tikzpicture}[scale=2.65]

\foreach \y in {0}\draw[->]({\y},-.25)--({\y},.625)node[above]{$\IM z$};
\foreach \y in {0} \draw[color=gray!20!white,fill,ultra thin] ({\y+1/2},0.5)--({\y+1/2},0)-- plot[samples=200,domain=0:101.5] ({5/12+cos(\x)/12},{sin(\x)/12})--plot[samples=360,domain=11.5:90] ({cos(\x)/sqrt(6)},{sin(\x)/sqrt(6)})-- ({\y},0.5);
\foreach \y in {0} \draw[color=gray!20!white,fill,ultra thin] ({\y},{1/sqrt(6)})--({\y},0)-- plot[samples=200,domain=0:180] ({-1/6+cos(\x)/6},{sin(\x)/6})--plot[samples=360,domain=0:79.5] ({-5/12+cos(\x)/12},{sin(\x)/12})--plot[samples=360,domain={180-11.5}:90] ({cos(\x)/sqrt(6)},{sin(\x)/sqrt(6)})-- ({\y},{1/sqrt(6)});

\foreach \y in {0}\draw[->]({\y-0.6},0)--({\y+0.6},0)node[right]{$\RE z$};

\draw[->,color=green,line width=1pt] (.5,0)--(.5,.4);
\draw[->,color=green,line width=1pt] (-.5,0)--(-.5,.4);
\draw[color=brown,line width=1pt] (0,0)--(0,.5);

\draw[->,color=brown,line width=.5pt] (0,.45)--(0,.44);
\draw[->,color=brown,line width=.5pt] (0,.25)--(0,.3);
\draw[color=green,line width=1pt] (.5,0)--(.5,.5);
\draw[color=green,line width=1pt] (-.5,0)--(-.5,.5);

\draw[dotted,thick] plot[samples=180,domain=90:180] ({cos(\x)/sqrt(6)},{sin(\x)/sqrt(6)});
\draw[dotted,thick] plot[samples=180,domain=0:90] ({cos(\x)/sqrt(6)},{sin(\x)/sqrt(6)});
\draw[color=green,line width=1pt] plot[samples=180,domain=0:180] ({1/6+cos(\x)/6},{sin(\x)/6});

\draw[color=green,line width=1pt] plot[samples=180,domain=0:180] ({-1/6+cos(\x)/6},{sin(\x)/6});

\draw[->,color=green,line width=1pt] plot[samples=20,domain=100:110] ({1/6+cos(\x)/6},{sin(\x)/6});
\draw[->,color=green,line width=1pt] plot[samples=20,domain=80:70] ({-1/6+cos(\x)/6},{sin(\x)/6});

\draw[dotted,thick] plot[samples=360,domain=0:180] ({5/12+cos(\x)/12},{sin(\x)/12});
\draw[dotted,thick] plot[samples=360,domain=0:180] ({-5/12+cos(\x)/12},{sin(\x)/12});
\draw[->,color=blue,line width=.5pt] plot[samples=20,domain=40:30] ({-5/12+cos(\x)/12},{sin(\x)/12});
\draw[->,color=blue,line width=.5pt] plot[samples=20,domain=115:125] ({-5/12+cos(\x)/12},{sin(\x)/12});
\draw[->,color=blue,line width=.5pt] plot[samples=20,domain=70:60] ({5/12+cos(\x)/12},{sin(\x)/12});
\draw[->,color=blue,line width=.5pt] plot[samples=20,domain=140:150] ({5/12+cos(\x)/12},{sin(\x)/12});
\draw[color=blue,line width=1pt] plot[samples=200,domain=0:180] ({5/12+cos(\x)/12},{sin(\x)/12});
\draw[color=blue,line width=1pt] plot[samples=200,domain=0:180] ({-5/12+cos(\x)/12},{sin(\x)/12});
\draw[->,color=magenta,line width=.5pt] plot[samples=20,domain=120:130] ({cos(\x)/sqrt(6)},{sin(\x)/sqrt(6)});
\draw[->,color=magenta,line width=.5pt] plot[samples=20,domain=60:50] ({cos(\x)/sqrt(6)},{sin(\x)/sqrt(6)});
\draw[color=magenta,line width=1pt] plot[samples=360,domain=11.5:168.5] ({cos(\x)/sqrt(6)},{sin(\x)/sqrt(6)});

\draw({1/2},0)node[below]{$^{\frac12} $};
\draw({1/3},0)node[below]{$^{\frac13} $};

\node[color=black] at (-0.375,.42) {$^{ \mathfrak D_{6,6}^-}$};

\node[color=black] at (0.15,.21) {$^{ \hat W_6\mathfrak D_{6,6}^-}$};
\node[color=black] at (-0.15,.21) {$^{ \hat W_6\mathfrak D_{6,6}^+}$};
\node[color=black] at (0.375,.42) {$^{ \mathfrak D_{6,6}^+}$};
\draw(0,-0.25) node[below]{\textbf{(f$^{\circ}{'}$)}};

\end{tikzpicture}\begin{tikzpicture}[scale=2.65]

\foreach \y in {0}\draw[->]({\y},-.25)--({\y},.625)node[above]{$\IM z$};

\foreach \y in {0} \draw[color=gray!20!white,fill,ultra thin] ({\y+1/2},.5)--({\y+1/2},0)--plot[samples=20,domain=0:180] ({5/12+cos(\x)/12},{sin(\x)/12})--plot[samples=360,domain=0:180] ({\y+1/6+cos(\x)/6},{sin(\x)/6})--({\y},{0})--({\y},.5);

\foreach \y in {0}\draw[->]({\y-0.6},0)--({\y+0.6},0)node[right]{$\RE z$};

\draw[dotted,thick] plot[samples=180,domain=90:180] ({.5+.5/sqrt(3)*cos(\x)},{.5/sqrt(3)*sin(\x)});
\draw[dotted,thick] plot[samples=180,domain=0:90] ({-.5+.5/sqrt(3)*cos(\x)},{.5/sqrt(3)*sin(\x)});
\draw[dotted,thick] plot[samples=360,domain=0:180] ({cos(\x)/6-1/6},{sin(\x)/6});
\draw[dotted,thick] plot[samples=360,domain=0:180] ({cos(\x)/6+1/6},{sin(\x)/6});
\draw[dotted,thick] ({1/3},0)--({1/3},{sqrt(2)/6});
\draw[dotted,thick] ({-1/3},0)--({-1/3},{sqrt(2)/6});
\draw[fill=yellow,draw=yellow]({1/3},{sqrt(2)/6})circle[radius=0.02];
\draw[fill=orange,draw=orange]({-1/3},{sqrt(2)/6})circle[radius=0.02];

\draw[color=yellow,line width=1pt] plot[samples=120,domain=90:150] ({.5+.5/sqrt(3)*cos(\x)},{.5/sqrt(3)*sin(\x)});
\draw[color=orange,line width=1pt] plot[samples=120,domain=30:90] ({-.5+.5/sqrt(3)*cos(\x)},{.5/sqrt(3)*sin(\x)});

\draw[color=blue,line width=1pt] (0,0)--(0,.5);
\draw[->,color=blue,line width=1pt] (0,0.3)--(0,.25);
\draw[->,color=blue,line width=.5pt] plot[samples=20,domain=70:60] ({5/12+cos(\x)/12},{sin(\x)/12});
\draw[->,color=blue,line width=.5pt] plot[samples=20,domain=130:140] ({-5/12+cos(\x)/12},{sin(\x)/12});
\draw[color=blue,line width=1pt] plot[samples=200,domain=0:180] ({5/12+cos(\x)/12},{sin(\x)/12});
\draw[color=blue,line width=1pt] plot[samples=200,domain=0:180] ({-5/12+cos(\x)/12},{sin(\x)/12});

\draw[->,color=green,line width=1pt] (.5,.5)--(.5,.35);
\draw[->,color=green,line width=1pt] (-.5,.5)--(-.5,.35);
\draw[color=green,line width=1pt] (.5,{.5/sqrt(3)})--(.5,.5);
\draw[color=green,line width=1pt] (-.5,{.5/sqrt(3)})--(-.5,.5);

\draw[color=red,line width=1pt] (.5,{.5/sqrt(3)})--(.5,0);
\draw[color=red,line width=1pt] (-.5,{.5/sqrt(3)})--(-.5,0);
\draw[->,color=red,line width=1pt] (.5,0)--(.5,.2);
\draw[->,color=red,line width=1pt] (-.5,0)--(-.5,.2);

\draw[->,color=orange,line width=.5pt] plot[samples=20,domain=35:45] ({-.5+.5/sqrt(3)*cos(\x)},{.5/sqrt(3)*sin(\x)});
\draw[->,color=orange,line width=.5pt] plot[samples=20,domain=80:70] ({-.5+.5/sqrt(3)*cos(\x)},{.5/sqrt(3)*sin(\x)});
\draw[->,color=yellow,line width=.5pt] plot[samples=20,domain=145:135] ({.5+.5/sqrt(3)*cos(\x)},{.5/sqrt(3)*sin(\x)});
\draw[->,color=yellow,line width=.5pt] plot[samples=20,domain=100:110] ({.5+.5/sqrt(3)*cos(\x)},{.5/sqrt(3)*sin(\x)});
\draw[->,color=red,line width=.5pt] plot[samples=20,domain=90:100] ({cos(\x)/6-1/6},{sin(\x)/6});
\draw[->,color=red,line width=.5pt] plot[samples=20,domain=100:90] ({cos(\x)/6+1/6},{sin(\x)/6});
\draw[color=red,line width=1pt] plot[samples=240,domain=0:120] ({cos(\x)/6-1/6},{sin(\x)/6});
\draw[color=red,line width=1pt] plot[samples=240,domain=60:180] ({cos(\x)/6+1/6},{sin(\x)/6});
\draw[->,color=green,line width=.5pt] plot[samples=20,domain=180:140] ({cos(\x)/6-1/6},{sin(\x)/6});
\draw[->,color=green,line width=.5pt] plot[samples=20,domain=0:40] ({cos(\x)/6+1/6},{sin(\x)/6});
\draw[color=green,line width=1pt] plot[samples=240,domain=180:120] ({cos(\x)/6-1/6},{sin(\x)/6});
\draw[color=green,line width=1pt] plot[samples=240,domain=0:60] ({cos(\x)/6+1/6},{sin(\x)/6});

\draw({1/2},0)node[below]{$^{\frac12} $};
\draw({1/3},0)node[below]{$^{\frac13} $};
\node[color=black] at (-0.375,.42) {$^{ \mathfrak D_{6,2}^-}$};

\node[color=black] at (0.375,.42) {$^{\mathfrak D_{6,2}^+}$};
\draw[<-,black,text=black,ultra thin] (-{5/12},.2)--(-{5/12},-.175) node at (-.4,-.23){$_{\hat W_2\hat\gamma\mathfrak D_{6,2}^-}$};
\draw[<-,black,text=black,ultra thin] ({5/12},.2)--({5/12},-.175) node at (.4,-.23){$_{\hat W_2\mathfrak D_{6,2}^+}$};
\draw(0,-0.25) node[below]{\textbf{(f$ ''$)}};\end{tikzpicture}\begin{tikzpicture}[scale=2.65]

\foreach \y in {0}\draw[->]({\y},-.25)--({\y},.625)node[above]{$\IM z$};
\foreach \y in {0} \draw[color=gray!20!white,fill,ultra thin] ({\y+1/2},0.5)--({\y+1/2},0)-- plot[samples=200,domain=0:101.5] ({5/12+cos(\x)/12},{sin(\x)/12})--plot[samples=360,domain=11.5:90] ({cos(\x)/sqrt(6)},{sin(\x)/sqrt(6)})-- ({\y},0.5);
\foreach \y in {0} \draw[color=gray!20!white,fill,ultra thin] ({\y},{1/sqrt(6)})--({\y},0)-- plot[samples=200,domain=180:0] ({1/6+cos(\x)/6},{sin(\x)/6})--plot[samples=360,domain=180:{180-79.5}] ({5/12+cos(\x)/12},{sin(\x)/12})--plot[samples=360,domain=11.5:90] ({cos(\x)/sqrt(6)},{sin(\x)/sqrt(6)})-- ({\y},{1/sqrt(6)});

\foreach \y in {0}\draw[->]({\y-0.6},0)--({\y+0.6},0)node[right]{$\RE z$};

\draw[->,color=green,line width=1pt] (.5,0)--(.5,.4);
\draw[->,color=green,line width=1pt] (-.5,0)--(-.5,.4);
\draw[color=blue,line width=1pt] (0,0)--(0,{1/sqrt(6)});

\draw[->,color=blue,line width=.5pt] (0,.3)--(0,.25);
\draw[color=green,line width=1pt] (.5,0)--(.5,.5);
\draw[color=green,line width=1pt] (-.5,0)--(-.5,.5);

\draw[color=green,line width=1pt] plot[samples=180,domain=0:180] ({1/6+cos(\x)/6},{sin(\x)/6});

\draw[color=green,line width=1pt] plot[samples=180,domain=0:180] ({-1/6+cos(\x)/6},{sin(\x)/6});

\draw[->,color=green,line width=1pt] plot[samples=20,domain=110:100] ({1/6+cos(\x)/6},{sin(\x)/6});
\draw[->,color=green,line width=1pt] plot[samples=20,domain=70:80] ({-1/6+cos(\x)/6},{sin(\x)/6});

\draw[dotted,thick] plot[samples=360,domain=0:180] ({5/12+cos(\x)/12},{sin(\x)/12});
\draw[dotted,thick] plot[samples=360,domain=0:180] ({-5/12+cos(\x)/12},{sin(\x)/12});
\draw[color=blue,line width=1pt] (0,0)--(0,.5);
\draw[->,color=blue,line width=1pt] (0,0.3)--(0,.25);
\draw[->,color=blue,line width=.5pt] plot[samples=20,domain=70:60] ({5/12+cos(\x)/12},{sin(\x)/12});
\draw[->,color=blue,line width=.5pt] plot[samples=20,domain=130:140] ({-5/12+cos(\x)/12},{sin(\x)/12});
\draw[color=blue,line width=1pt] plot[samples=200,domain=0:180] ({5/12+cos(\x)/12},{sin(\x)/12});
\draw[color=blue,line width=1pt] plot[samples=200,domain=0:180] ({-5/12+cos(\x)/12},{sin(\x)/12});

\draw[color=green,line width=1pt] (.5,{.5/sqrt(3)})--(.5,.5);
\draw[color=green,line width=1pt] (-.5,{.5/sqrt(3)})--(-.5,.5);
\draw[->,color=magenta,line width=.5pt] plot[samples=20,domain=60:50] ({-1/4+cos(\x)/4},{sin(\x)/4});
\draw[->,color=magenta,line width=.5pt] plot[samples=20,domain=140:150] ({-1/4+cos(\x)/4},{sin(\x)/4});
\draw[->,color=yellow,line width=.5pt] plot[samples=20,domain=120:130] ({1/4+cos(\x)/4},{sin(\x)/4});
\draw[->,color=yellow,line width=.5pt] plot[samples=20,domain=40:30] ({1/4+cos(\x)/4},{sin(\x)/4});
\draw[color=yellow,line width=1pt] plot[samples=360,domain=0:180] ({1/4+cos(\x)/4},{sin(\x)/4});
\draw[color=magenta,line width=1pt] plot[samples=360,domain=0:180] ({-1/4+cos(\x)/4},{sin(\x)/4});

\draw[dotted,thick] ({1/3},0)--({1/3},{sqrt(2)/6});
\draw[dotted,thick] ({-1/3},0)--({-1/3},{sqrt(2)/6});
\draw[fill=yellow,draw=yellow]({1/3},{sqrt(2)/6})circle[radius=0.02];
\draw[fill=magenta,draw=magenta]({-1/3},{sqrt(2)/6})circle[radius=0.02];
\draw({1/2},0)node[below]{$^{\frac12} $};
\draw({1/3},0)node[below]{$^{\frac13} $};

\node[color=black] at (-0.375,.42) {$^{\mathfrak D_{6,2}^-}$};

\draw[<-,black,text=black,ultra thin] (.3,.2)--(.3,-.175) node at (0.25,-.23) {$^{\hat W_2 \mathfrak D_{6,2}^+}$};
\draw[<-,black,text=black,ultra thin] (-.3,.2)--(-.3,-.175) node at  (-0.25,-.23) {$^{ \hat W_2\hat\gamma\mathfrak D_{6,2}^-}$};
\node[color=black] at (0.375,.42) {$^{ \mathfrak D_{6,2}^+}$};
\draw(0,-0.25) node[below]{\textbf{(f$'''$)}};

\end{tikzpicture}
\end{center}\end{figure*}
\vspace{-2em}
\noindent Here, panel \textit{f}$ ^\circ{'}$ issues from panel \textit{f}$ ^\circ$, and may be compared to panel \textit{d}$'$ above. For the design of $ \mathfrak D_{6,2}$ in panel \textit{f}$'' $, its interior $ \mathring{ \mathfrak D}_{6,2}$ coincides with $\mathring{ \mathfrak D}_{6,3}  $ (see \cite[Fig.\ 2\textit{e}]{Zhou2018ExpoDESY} and panel \textit{e}$'$ above), but the boundary sides $\partial \mathfrak D_{6,2}$ and $\partial \mathfrak D_{6,3}$ are glued in different ways. The design of $ \mathfrak D_{6,2}$ in panel \textit{f}$''' $ is due to  Chan--Verrill \cite[Fig.\ 1]{ChanVerrill2009}.

In certain arithmetic applications (not quite related to the current survey article), it is desirable to align the boundary sides   $\partial \mathfrak D_{6,k}$ with special geodesics in some $ SL(2,\mathbb Z)$-images of $ \mathfrak D$ \cite[Fig.~2\textit{a}]{Zhou2018ExpoDESY}, the canonical choice of the fundamental domain for $ \Gamma_0(1)=SL(2,\mathbb Z)$. When such requirements become mandatory, we are left with fewer options for the design of the fundamental domains, such as $ \mathfrak D_{6,6}$ in panel \textit{d}, $\mathfrak D_{6,3}^+ \cup\hat W_3\hat T\mathfrak D_{6,3}^-$ in panel \textit{e}$'$, and $ \mathfrak D_{6,2}^+\cup\mathfrak D_{6,2}^-$ in panel \textit{f}$'''$.

If we restructure the fundamental domain $ \mathfrak D_{6,k}$ according to the last paragraph, then the Klein $j$-invariant $j(z)$ will become real-valued on $ \partial \mathfrak D_{6,k}$. In the following illustrations, we subdivide $ \mathfrak D_{6,k}$ by the sign of $ \IM j(z)$, assigning gray color to regions where $ \IM j(z)>0$:

\vspace{-1em}
\begin{figure*}[h!]
\begin{center}
\begin{tikzpicture}[scale=2]

\foreach \y in {0}\draw[->]({\y},-.25)--({\y},1.625)node[above]{$\IM z$};
\foreach \y in {0} \draw[color=gray!40!white,fill,ultra thin] plot[samples=60,domain=60:90] ({\y+cos(\x)},{sin(\x)})--({\y},0)--plot[samples=120,domain=180:120] ({\y+1+cos(\x)},{sin(\x)});


\foreach \y in {0} \draw[color=gray!40!white,fill,ultra thin] plot[samples=180,domain=90:180] ({\y+1/2+cos(\x)/2},{sin(\x)/2})--plot[samples=240,domain=180:60] ({\y+1/3+cos(\x)/3},{sin(\x)/3});


\foreach \y in {0} \draw[color=gray!40!white,fill,ultra thin] plot[samples=120,domain=120:143.1] ({\y+2/3+cos(\x)/3},{sin(\x)/3})--plot[samples=360,domain=53:0] ({\y+1/4+cos(\x)/4},{sin(\x)/4})--({\y+1/2},{1/2/sqrt(3)});

\foreach \y in {0} \draw[color=gray!40!white,fill,ultra thin] plot[samples=360,domain=180:53] ({\y+1/4+cos(\x)/4},{sin(\x)/4})--plot[samples=120,domain=143.1:158.2] ({\y+2/3+cos(\x)/3},{sin(\x)/3})--plot[samples=360,domain=38.2:180] ({\y+1/5+cos(\x)/5},{sin(\x)/5});




\foreach \y in {0} \draw[color=gray!40!white,fill,ultra thin] plot[samples=120,domain=0:98.2] ({\y+3/8+cos(\x)/8},{sin(\x)/8})--plot[samples=360,domain=38.2:22.6] ({\y+1/5+cos(\x)/5},{sin(\x)/5})--plot[samples=120,domain=112.6:0] ({\y+5/12+cos(\x)/12},{sin(\x)/12});

\foreach \y in {0} \draw[color=gray!40!white,fill,ultra thin] plot[samples=120,domain=98.2:126.9] ({\y+3/8+cos(\x)/8},{sin(\x)/8})--plot[samples=360,domain=36.9:0] ({\y+1/6+cos(\x)/6},{sin(\x)/6})--plot[samples=120,domain=180:158.2] ({\y+2/3+cos(\x)/3},{sin(\x)/3});


\foreach \y in {0}
\draw[->]({\y-0.15},0)--({\y+0.6},0)node[right]{$\RE z$};

\draw[->,color=green,line width=1pt] (.5,0)--(.5,.8);
\draw[->,color=green,line width=1pt] (.5,0)--(.5,.75);
\draw[color=green,line width=1pt] (.5,0)--(.5,1.5);
\draw[->,color=blue,line width=.1pt] plot[samples=20,domain=60:50] ({5/12+cos(\x)/12},{sin(\x)/12});
\draw[->,color=blue,line width=.1pt] plot[samples=20,domain=140:150] ({5/12+cos(\x)/12},{sin(\x)/12});
\draw[color=blue,line width=1pt] plot[samples=360,domain=0:180] ({cos(\x)/12+5/12},{sin(\x)/12});
\draw[color=brown,line width=1pt] (0,0)--(0,1.5);
\draw[->,color=green,line width=.5pt] plot[samples=20,domain=80:90] ({cos(\x)/6+1/6},{sin(\x)/6});
\draw[->,color=green,line width=.5pt] plot[samples=20,domain=90:100] ({cos(\x)/6+1/6},{sin(\x)/6});
\draw[color=green,line width=1pt] plot[samples=360,domain=0:180] ({cos(\x)/6+1/6},{sin(\x)/6});
\draw[>-,color=brown,line width=1pt] (0,0.8)--(0,0);
\draw[>-,color=brown,line width=1pt] (0,0.75)--(0,0);
\draw[>-,color=brown,line width=1pt] (0,0.7)--(0,0);
\draw[>-,color=brown,line width=1pt] (0,0.2)--(0,0.5);
\draw[>-,color=brown,line width=1pt] (0,0.25)--(0,.5);
\draw[>-,color=brown,line width=1pt] (0,0.15)--(0,.5);

\draw[dotted,thick] plot[samples=180,domain=0:90] ({cos(\x)/sqrt(6)},{sin(\x)/sqrt(6)});

\draw({1/2},0)node[below]{$^{\frac12} $};
\draw({1/3},0)node[below]{$^{\frac13} $};
\draw[fill=black](0,{1/sqrt(6)})circle[radius=0.02]node at (-.1,{1/sqrt(6)-.02}) {$_{\frac{i}{\sqrt6}}$};
\draw[fill=black]({2/5},{1/5/sqrt(6)})circle[radius=0.02];

\node[color=black] at (0.375,1.375) {$^{ \mathfrak D_{6,6}}$};
\draw(0.25,-0.25) node[below]{\textbf{(d$^*$)}};

\end{tikzpicture}\begin{tikzpicture}[scale=2]

\foreach \y in {0}\draw[->]({\y},-.25)--({\y},1.625)node[above]{$\IM z$};
\foreach \y in {0} \draw[color=gray!40!white,fill,ultra thin] plot[samples=60,domain=60:90] ({\y+cos(\x)},{sin(\x)})--({\y},0)--plot[samples=120,domain=180:120] ({\y+1+cos(\x)},{sin(\x)});


\foreach \y in {0} \draw[color=gray!40!white,fill,ultra thin] plot[samples=180,domain=90:180] ({\y+1/2+cos(\x)/2},{sin(\x)/2})--plot[samples=240,domain=180:60] ({\y+1/3+cos(\x)/3},{sin(\x)/3});


\foreach \y in {0} \draw[color=gray!40!white,fill,ultra thin] plot[samples=120,domain=120:143.1] ({\y+2/3+cos(\x)/3},{sin(\x)/3})--plot[samples=360,domain=53:0] ({\y+1/4+cos(\x)/4},{sin(\x)/4})--({\y+1/2},{1/2/sqrt(3)});

\foreach \y in {0} \draw[color=gray!40!white,fill,ultra thin] plot[samples=360,domain=180:53] ({\y+1/4+cos(\x)/4},{sin(\x)/4})--plot[samples=120,domain=143.1:158.2] ({\y+2/3+cos(\x)/3},{sin(\x)/3})--plot[samples=360,domain=38.2:180] ({\y+1/5+cos(\x)/5},{sin(\x)/5});




\foreach \y in {0} \draw[color=gray!40!white,fill,ultra thin] plot[samples=120,domain=0:98.2] ({\y+3/8+cos(\x)/8},{sin(\x)/8})--plot[samples=360,domain=38.2:22.6] ({\y+1/5+cos(\x)/5},{sin(\x)/5})--plot[samples=120,domain=112.6:0] ({\y+5/12+cos(\x)/12},{sin(\x)/12});

\foreach \y in {0} \draw[color=gray!40!white,fill,ultra thin] plot[samples=120,domain=98.2:126.9] ({\y+3/8+cos(\x)/8},{sin(\x)/8})--plot[samples=360,domain=36.9:0] ({\y+1/6+cos(\x)/6},{sin(\x)/6})--plot[samples=120,domain=180:158.2] ({\y+2/3+cos(\x)/3},{sin(\x)/3});


\foreach \y in {0}
\draw[->]({\y-0.15},0)--({\y+0.6},0)node[right]{$\RE z$};

\draw[->,color=green,line width=1pt] (.5,0)--(.5,.6);
\draw[->,color=green,line width=1pt] (.5,.2)--(.5,.1);
\draw[color=green,line width=1pt] (.5,0)--(.5,1.5);
\draw[color=blue,line width=1pt] (0,0)--(0,1.5);
\draw[->,color=blue,line width=1pt] (0,0.6)--(0,.5);

\draw[->,color=blue,line width=.5pt] plot[samples=20,domain=90:100] ({5/12+cos(\x)/12},{sin(\x)/12});
\draw[color=blue,line width=1pt] plot[samples=200,domain=0:180] ({5/12+cos(\x)/12},{sin(\x)/12});

\draw[->,color=red,line width=.5pt] plot[samples=20,domain=100:90] ({cos(\x)/6+1/6},{sin(\x)/6});
\draw[->,color=red,line width=.5pt] plot[samples=20,domain=30:40] ({cos(\x)/6+1/6},{sin(\x)/6});

\draw[color=red,line width=1pt] plot[samples=240,domain=0:180] ({cos(\x)/6+1/6},{sin(\x)/6});

\draw[dotted,thick] plot[samples=180,domain=90:180] ({.5+cos(\x)/2/sqrt(3)},{sin(\x)/2/sqrt(3)});

\draw({1/2},0)node[below]{$^{\frac12} $};
\draw({1/3},0)node[below]{$^{\frac13} $};
\draw[fill=black](0.5,{1/2/sqrt(3)})circle[radius=0.02]node at (.7,{1/2/sqrt(3)-.02}) {$_{\frac12+\frac{i}{2\sqrt3}}$};
\draw[fill=black](0.25,{1/4/sqrt(3)})circle[radius=0.02];

\node[color=black] at (0.375,1.375) {$^{ \mathfrak D_{6,3}}$};
\draw(0.25,-0.25) node[below]{\textbf{(e$^*$)}};

\end{tikzpicture}\begin{tikzpicture}[scale=2]

\foreach \y in {0}\draw[->]({\y},-.25)--(0,0);
\foreach \y in {0}\draw[->]({\y},1.5)--({\y},1.625)node[above]{$\IM z$};
\foreach \y in {0} \draw[color=gray!40!white,fill,ultra thin] plot[samples=60,domain=60:90] ({\y+cos(\x)},{sin(\x)})--({\y},0)--plot[samples=120,domain=180:120] ({\y+1+cos(\x)},{sin(\x)});

\foreach \y in {-1}  \draw[color=gray!40!white,fill,ultra thin] plot[samples=120,domain=60:0] ({\y+cos(\x)},{sin(\x)})--plot[samples=180,domain=0:90] ({\y+1/2+cos(\x)/2},{sin(\x)/2});

\foreach \y in {0} \draw[color=gray!40!white,fill,ultra thin] plot[samples=180,domain=90:180] ({\y+1/2+cos(\x)/2},{sin(\x)/2})--plot[samples=240,domain=180:60] ({\y+1/3+cos(\x)/3},{sin(\x)/3});

\foreach \y in {-1} \draw[color=gray!40!white,fill,ultra thin] plot[samples=120,domain=60:36.9] ({\y+1/3+cos(\x)/3},{sin(\x)/3})--plot[samples=360,domain=127:0] ({\y+3/4+cos(\x)/4},{sin(\x)/4})--plot[samples=240,domain=0:120] ({\y+2/3+cos(\x)/3},{sin(\x)/3});

\foreach \y in {0} \draw[color=gray!40!white,fill,ultra thin] plot[samples=120,domain=120:143.1] ({\y+2/3+cos(\x)/3},{sin(\x)/3})--plot[samples=360,domain=53:0] ({\y+1/4+cos(\x)/4},{sin(\x)/4})--({\y+1/2},{1/2/sqrt(3)});

\foreach \y in {0} \draw[color=gray!40!white,fill,ultra thin] ({\y},1.5)--({\y},1)--plot[samples=60,domain=90:120] ({\y+cos(\x)},{sin(\x)})--({\y-.5},{sqrt(3)/2})--({\y-.5},1.5);

\foreach \y in {0}
\draw[->]({\y-0.6},0)--({\y+0.6},0)node[right]{$\RE z$};

\draw[->,color=green,line width=1pt] (.5,0)--(.5,.6);
\draw[->,color=green,line width=1pt] (-.5,0)--(-.5,.6);

\draw[color=green,line width=1pt] (.5,{.5/sqrt(3)})--(.5,1.5);
\draw[color=green,line width=1pt] (-.5,{.5/sqrt(3)})--(-.5,1.5);
\draw[->,color=magenta,line width=.5pt] plot[samples=20,domain=60:50] ({-1/4+cos(\x)/4},{sin(\x)/4});
\draw[->,color=magenta,line width=.5pt] plot[samples=20,domain=140:150] ({-1/4+cos(\x)/4},{sin(\x)/4});
\draw[->,color=yellow,line width=.5pt] plot[samples=20,domain=120:130] ({1/4+cos(\x)/4},{sin(\x)/4});
\draw[->,color=yellow,line width=.5pt] plot[samples=20,domain=40:30] ({1/4+cos(\x)/4},{sin(\x)/4});
\draw[color=yellow,line width=1pt] plot[samples=360,domain=0:180] ({1/4+cos(\x)/4},{sin(\x)/4});
\draw[color=magenta,line width=1pt] plot[samples=360,domain=0:180] ({-1/4+cos(\x)/4},{sin(\x)/4});

\draw[dotted,thick] ({1/3},0)--({1/3},{sqrt(2)/6});
\draw[dotted,thick] ({-1/3},0)--({-1/3},{sqrt(2)/6});
\draw[fill=yellow,draw=yellow]({1/3},{sqrt(2)/6})circle[radius=0.02];
\draw[fill=magenta,draw=magenta]({-1/3},{sqrt(2)/6})circle[radius=0.02];

\draw({1/2},0)node[below]{$^{\frac12} $};
\draw({1/3},0)node[below]{$^{\frac13} $};

\node[color=black] at (0.375,1.375) {$^{ \mathfrak D_{6,2}}$};
\draw(0,-0.25) node[below]{\textbf{(f$^*$)}};

\end{tikzpicture}
\end{center}
\end{figure*}
\vspace{-2em}
\noindent so that the relation  $ [\Gamma_{0}(6)_{+k}:\Gamma_{0}(6)]=2$ is graphically obvious, upon comparison to \cite[Fig.~2\textit{c}]{Zhou2018ExpoDESY}.

In practice, we chose panel \textit{e} over panel \textit{e}$^*$ for the design of $ \mathfrak D_{6,3}$, because the geodesic joining $ \frac12+\frac{i}{2\sqrt3}$ to $ \frac14+\frac{i}{4\sqrt3}$ was essential to some computations in \cite[\S5]{Zhou2017WEFa}, thus deserving a place on $ \partial \mathfrak D_{6,3}$.

The current arXiv version of \cite{Zhou2018ExpoDESY} (\url{arXiv:1801.05555v3}) contains updated panels \textit{d} and \textit{f} in the main text, along with an erratum section at its end.

\end{document}